\documentclass[12pt,twoside]{article}
\usepackage{amsthm}
\usepackage{epsfig}
\usepackage{enumitem}
    \usepackage{ amssymb }
\usepackage{amsmath}
\usepackage{mathtools}
\usepackage{multirow}
\usepackage{float}
\usepackage{epsfig}
\input epsf

\newtheorem{thm}{Theorem}
\newtheorem{defi}[thm]{Definition}
\newtheorem{lem}[thm]{Lemma}
\newtheorem{claim}[thm]{Claim}
\newtheorem{cor}[thm]{Corollary}

\newtheorem{obs}[thm]{Observation}

\newcommand{\ex}{{\rm ex}}

\newcommand{\AAA}{{\mathbf A}}
\newcommand{\BB}{{\mathbf B}}
\newcommand{\CC}{{\mathbf C}}
\newcommand{\DD}{{\mathbf D}}
\newcommand{\FFF}{{\mathbf F}}

\newcommand{\EE}{{\mathcal E}}
\newcommand{\FF}{{\mathcal F}}
\newcommand{\HH}{{\mathcal H}}
\newcommand{\PP}{{\mathcal P}}
\newcommand{\QQ}{{\mathcal Q}}
\newcommand{\RR}{{\mathcal R}}
\newcommand{\TT}{{\mathcal T}}

\newcommand{\nice}{{near optimal\enskip }}
\newcommand{\niceness}{{near optimality\enskip }}
\newcommand{\eps}{\varepsilon}

\begin{document}

\title{\huge Triangle-free triple systems}
\author{{\bf Peter Frankl},  ${}^{1}$ \,
{\bf Zolt\'an F\"uredi}, ${}^{2}$ \, {\bf Ido Goorevitch}, ${}^{3}$ \\ {\bf Ron Holzman}, ${}^{4}$ \, and
\, {\bf G\'abor Simonyi}  ${}^{5}$
}
\date{
${}^{1,2,5}$ HUN-REN  Alfr\'ed R\'enyi Institute of Mathematics, \\
Budapest, P. O. Box 127, 1364 Hungary\\
${}^{5}$ Department of Computer Science and Information Theory,\\
Faculty of Electrical Engineering and Informatics,\\
Budapest University of Technology and Economics, Budapest, Hungary\\
${}^{3,4}$ Department of Mathematics,\\ Technion-Israel Institute of Technology,\\ Haifa 3200003, Israel
\\ ${}$
\\
E-mails: {\tt frankl@renyi.hu}, \quad {\tt furedi@renyi.hu}, 
\quad {\tt idogoorevitch@gmail.com}, \quad {\tt holzman@technion.ac.il}, \quad and  \quad {\tt simonyi@renyi.hu}
\thanks{Research
of the second author was supported in part by the National Research Development and Innovation Office, NKFIH,
grants K--132696 and KKP 133819.
\\ \noindent
Research of the fifth author was partially supported by the National Research, Development and Innovation Office (NKFIH) grants K--132696 and SNN--135643 of NKFIH Hungary.
}
}



\maketitle

\begin{abstract}
{\small
There are four non-isomorphic configurations of triples that can form a
triangle in a $3$-uniform hypergraph.
Forbidding different combinations of these four
configurations, fifteen extremal problems can be defined, several of which
already appeared in the literature in some different context.
Here we systematically study all of these problems solving the new cases
exactly or asymptotically. In many cases we also characterize the extremal constructions.}

\end{abstract}

\section{Introduction}
\message{Introduction}

\subsection{Notations}\label{ss:10}
A {\em hypergraph} $\FF$ is a pair $(V(\FF),E(\FF))$
 where the set of {\em edges}, $E(\FF)$, is a family of subsets of
 the set of {\em vertices} $V(\FF)$  (see Berge~\cite{Berge}, Lov\'asz~\cite{LovEx}).
To simplify notations we frequently identify $\FF$ with $E(\FF)$ (and omit parentheses, etc.) if it does not cause ambiguity.
If all edges of $\FF$ have the same size, say $r$, then we speak of an $r$-graph, or $r$-{\em uniform} hypergraph.
With this terminology  $2$-graphs are simple graphs without loops and 3-uniform hypergraphs are also called {\em triple systems}.
We usually identify $V=V(\FF)$ with the set of first $n$ positive integers, $[n]:=\{ 1, 2, 3, \dots, n\}$. Then ${V \choose r}$ stands for the complete $r$-graph with vertex set $V$, it is $\{ F: |F|=r, F\subseteq V \}$.



We write $\partial \FF$ for $\{ F\setminus \{x\}: F\in \FF, x\in F  \}$, it is called the {\em shadow} of $\FF$.
If $\FF$ is a triple system then its shadow is a simple graph.
The {\em degree} of $S\subseteq V(\FF)$ is the number of edges containing $S$, in notation
$$\deg_\FF(S):= |\{ F: F\in \FF, S\subseteq F\}|.
 $$
Given a triple system $\FF$ we write $\FF[x]$ for the \emph{link} of $x$ in $\FF$, i.e., the graph on $V(\FF) \setminus \{x\}$ whose edges are those pairs that form with $x$ a triple in $\FF$. 
If $F$ is an edge of a $3$-uniform hypergraph ${\FF}$ such that $x,y\in F$, $x \ne y$,
and there is no other $F'\in {\FF}$ containing both $x$ and $y$, then we say that
the pair $xy$ is an {\it own pair} of $F$. With our previous notation, $\deg_\FF(xy)=1$.

\subsection{Cycles in hypergraphs}
In hypergraph theory
 a cycle of length $\ell$ is usually defined as a sequence of
 {\em distinct}
 vertices $x_1, x_2, \dots, x_\ell$ together with a sequence of
 {\em distinct} edges
 $E_1, \dots, E_\ell$ such that $E_i$ contains $x_i$ and $x_{i+1}$
 for $1\leq i < \ell$ and $E_\ell$ contains $x_\ell$ and $x_1$.
We refer to a cycle of $\ell$ edges as an $\ell$-cycle.
For ordinary graphs, the above definition coincides with the definition
 of a cycle $C_\ell$ in graphs.
We call a 3-cycle a {\it triangle}.

One of the oldest problems in extremal graph theory is to
 investigate the maximum number of edges of graphs without given short
 cycles, e.~g., the Tur\'an-Mantel theorem says that a triangle-free
 graph on $n$ vertices has at most
 $\lfloor \frac{1}{4}n^2 \rfloor$ edges and the only extremal graph is the
 complete bipartite graph with parts of sizes
 $\lfloor \frac{n}{2} \rfloor$ and $\lceil \frac{n}{2} \rceil$, respectively.
For further graph results see the book of Bollob\'as~\cite{Bol}.

Hypergraph problems are 
 more complicated, but there are some results, e.~g.,
 that of Lazebnik and Verstra\"ete~\cite{LazVer}
 saying that if $\FF$ is $3$-uniform, has $n$ vertices, girth five
 (i.e., no cycles of length less than five) and a maximum
 number of edges, then $|E(\FF)|=\frac{1}{6}n^{3/2}+ o(n^{3/2})$.

In this paper we deal with 3-uniform hypergraphs containing no triangles of certain types.

\subsection{Triangles in triple systems}
One can easily realize that
there are
altogether four non-isomorphic configurations of three triples that create a
triangle. We will call these configurations $\AAA,\BB,\CC$, and $\DD$ (see Figure~1 below). The main goal of this paper is to 
 study the maximum number of triples one can have on $n$ vertices while excluding certain combinations of the four triangle creating
 configurations.

Configuration $\AAA$ is the one needing six vertices, that is, it is isomorphic to
the hypergraph with $V=\{1,\dots, 6\}, E=\{124,135,236\}.$

Configuration $\BB$ is on five vertices, where two of the three triples involved
intersect in two vertices and each of them meets the third triple in a different vertex.
Thus $\BB$ is
isomorphic to the hypergraph with $V=\{1,\dots, 5\},
E=\{125,134,234\}.$

Configuration $\CC$ is the one on four vertices, that is, it is isomorphic to
the hypergraph with $V=\{1,\dots, 4\}, E=\{124,134,234\}.$

Configuration $\DD$ is on five vertices again, with one triple sharing with each of the other two a different pair of vertices. It is isomorphic to the hypergraph
with $V=\{1,\dots, 5\}, E=\{123,134,235\}.$


\begin{center}
\begin{figure}[H]
\epsfig{file=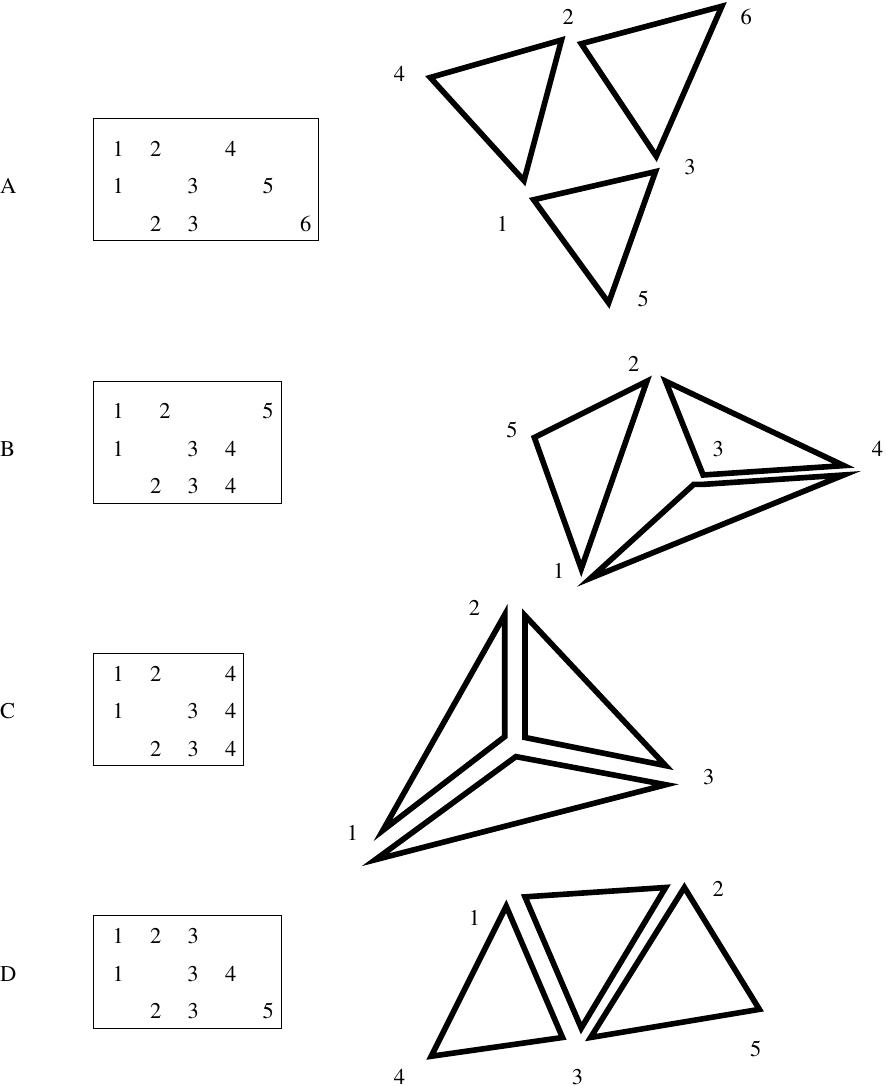,height=11cm}
\caption{The four configurations of three triples forming a triangle}
\end{figure}
\end{center}

Our starting point is a result (first proved by Gy\H{o}ri~\cite{GyE}) giving the maximum number of edges in a
triangle-free $3$-uniform hypergraph, i.e., in case all of $\AAA,\BB,\CC$
and $\DD$ are forbidden subconfigurations.
In Section~\ref{sec:all} we reprove it and also describe the extremal families.
Then we consider the fourteen other cases when some proper subset of
 $\{\AAA,\BB,\CC,\DD\}$ is forbidden. Five of these subsets define
problems of long history, most of them are solved, but not all. 
Then we go on to
 investigate the other nine. For all but one of those nine combinations, we obtain an exact solution of the extremal problem (though sometimes only for large enough $n$), often along with a characterization of the extremal constructions. In the one exceptional case we give an asymptotic solution.

To state our results let us introduce the notation $\ex(n, X_1\dots X_k)$
for $\{X_1,\dots, X_k\} \subseteq \{\AAA,\BB,\CC,\DD\}$.
It denotes the maximum number of triples in a $3$-uniform hypergraph on $n$ vertices which does not contain any of the configurations $X_1, \dots, X_k$.
 Table 1 below summarizes the results, including earlier ones and those in the present paper.

\newpage


\begin{center}
\begin{table}[H]
    \begin{tabular}{ccccccc}
    & & $\AAA\BB$ & & & & \\ \cline{3-3}
    & & \multicolumn{1}{|c|}{${n-1 \choose 2}$ for $n \ge 5$} & & & &\\
    & & \multicolumn{1}{|c|}{and char. ext.} & & & &\\
    & & \multicolumn{1}{|c|}{Theorem~\ref{thm:AB}
    } & & & &\\ \cline{3-3}\noalign{\smallskip}\noalign{\smallskip}\noalign{\smallskip}
        $\AAA$ & & $\AAA\CC$ & & $\AAA\BB\CC$ & & \\ \cline{1-1} \cline{3-3} \cline{5-5}
    \multicolumn{1}{|c|}{${n-1 \choose 2}$ for $n \ge n_0$} & & \multicolumn{1}{|c|}{$\lfloor \frac{1}{4}(n-1)^2 \rfloor$ for even $n$} & & \multicolumn{1}{|c|}{$\lfloor \frac{1}{4}(n-1)^2 \rfloor$} & & \\
        \multicolumn{1}{|c|}{and char. ext.} & & \multicolumn{1}{|c|}{$\frac{1}{4}(n-1)^2 + 1$ for odd $n \ge 5$} & & \multicolumn{1}{|c|}{and char. ext.} & & \\
    \multicolumn{1}{|c|}{Frankl\&F\"uredi~\cite{FF87}} & & \multicolumn{1}{|c|}{and char. ext.} & &\multicolumn{1}{|c|}{Theorem~\ref{thm:ABC}} & & \\
    \multicolumn{1}{|c|}{} & & \multicolumn{1}{|c|}{Theorem~\ref{thm:AC}} & & \multicolumn{1}{|c|}{} & & \\
    \cline{1-1} \cline{3-3} \cline{5-5}\noalign{\smallskip}\noalign{\smallskip}\noalign{\smallskip}
    $\BB$ & & $\AAA\DD$ & & $\AAA\BB\DD$ & & \\ \cline{1-1} \cline{3-3} \cline{5-5}
    \multicolumn{1}{|c|}{$\lfloor \frac{n}{3} \rfloor \lfloor \frac{n+1}{3} \rfloor \lfloor \frac{n+2}{3} \rfloor$ for $n \ge n_0$} & & \multicolumn{1}{|c|}{$\lfloor \frac{1}{4}(n-1)^2 \rfloor + 2$ for even $n \ge 4$} & & \multicolumn{1}{|c|}{$\lfloor \frac{1}{8}n^2 \rfloor$ for $n \ge 8$} & & \\
    \multicolumn{1}{|c|}{Frankl\&F\"uredi~\cite{FF83}} & & \multicolumn{1}{|c|}{$\frac{1}{4}(n-1)^2$ for odd $n$} & & \multicolumn{1}{|c|}{and char. ext.} & & \\
    \multicolumn{1}{|c|}{char. ext.} & & \multicolumn{1}{|c|}{and char. ext.} & & \multicolumn{1}{|c|}{{Theorem~\ref{thm:ABD}}} & & $\AAA\BB\CC\DD$ \\ \cline{7-7}
    \multicolumn{1}{|c|}{Keevash\&Mubayi~\cite{KeeMub}} & & \multicolumn{1}{|c|}{Theorem~\ref{thm:AD}} & & \multicolumn{1}{|c|}{} & & \multicolumn{1}{|c|}{$\lfloor \frac{1}{8}n^2 \rfloor$} \\ \cline{1-1} \cline{3-3} \cline{5-5}
    & & & & & & \multicolumn{1}{|c|}{Gy\H{o}ri~\cite{GyE}} \\
    $\CC$ & & $\BB\CC$ & & $\AAA\CC\DD$ & & \multicolumn{1}{|c|}{char. ext.} \\ \cline{1-1} \cline{3-3} \cline{5-5}
    \multicolumn{1}{|c|}{$\ge \frac{2}{7}{n \choose 3}\bigl( 1 - o(1) \bigr)$} & & \multicolumn{1}{|c|}{$\lfloor \frac{n}{3} \rfloor \lfloor \frac{n+1}{3} \rfloor \lfloor \frac{n+2}{3} \rfloor$} & & \multicolumn{1}{|c|}{$\lfloor \frac{1}{4}(n-1)^2 \rfloor$} & & \multicolumn{1}{|c|}{Theorem~\ref{thm:all}} \\ \cline{7-7}
    \multicolumn{1}{|c|}{Frankl\&F\"uredi~\cite{FF84}} & & \multicolumn{1}{|c|}{and char. ext.} & & \multicolumn{1}{|c|}{and char. ext.} & & \\
    \multicolumn{1}{|c|}{$\le .287{n \choose 3}$ for $n \ge n_0$} & & \multicolumn{1}{|c|}{Bollob\'as~\cite{BB}} & & \multicolumn{1}{|c|}{Corollary~\ref{cor:ACD}} & & \\
    \multicolumn{1}{|c|}{Falgas-Ravry\&Vaughan~\cite{FRV13}} & & \multicolumn{1}{|c|}{} & & \multicolumn{1}{|c|}{} & & \\ \cline{1-1} \cline{3-3} \cline{5-5}\noalign{\smallskip}\noalign{\smallskip}\noalign{\smallskip}
        $\DD$ & & $\BB\DD$ & & $\BB\CC\DD$ & & \\ \cline{1-1} \cline{3-3} \cline{5-5}
    \multicolumn{1}{|c|}{$\frac{n(n-1)}{3}$, $n \equiv 1,4$} & & \multicolumn{1}{|c|}{$\frac{n(n-1)}{3}$, $n \equiv 1,4$} & & \multicolumn{1}{|c|}{$\frac{1}{4}n^2\bigl( 1 - o(1) \bigr)$} & & \\
    \multicolumn{1}{|c|}{$\frac{n(n-1)}{3} - 4$, $n \equiv 7,10$} & & \multicolumn{1}{|c|}{$\frac{n(n-1)}{3} - 4$, $n \equiv 7,10$} & & \multicolumn{1}{|c|}{Theorem~\ref{thm:BCD}} & & \\
    \multicolumn{1}{|c|}{$\frac{n(n-2)}{3}$, $n \equiv 0,2,3,6,8,9$} & & \multicolumn{1}{|c|}{$\frac{n(n-2)}{3}$, $n \equiv 0,2,3,8$} & & \multicolumn{1}{|c|}{} & & \\
    \multicolumn{1}{|c|}{$\frac{n(n-2)}{3} - 1$, $n \equiv 5,11$} & & \multicolumn{1}{|c|}{$\frac{n(n-2)}{3} - 1$, $n \equiv 5,6,9,11$} & & \multicolumn{1}{|c|}{} & & \\
    \multicolumn{1}{|c|}{$n \ge n_0$ (residues mod $12$)} & & \multicolumn{1}{|c|}{$n \ge n_0$ (residues mod $12$)} & & \multicolumn{1}{|c|}{} & & \\
    \multicolumn{1}{|c|}{Theorem~\ref{thm:DBD}} & & \multicolumn{1}{|c|}{Theorem~\ref{thm:DBD}} & & \multicolumn{1}{|c|}{} & & \\ \cline{1-1} \cline{3-3} \cline{5-5}\noalign{\smallskip}\noalign{\smallskip}\noalign{\smallskip}
    & & $\CC\DD$ & & & & \\ \cline{3-3}
    & & \multicolumn{1}{|c|}{$\lfloor \frac{1}{4}(n-1)^2 \rfloor$} & & & & \\
        & & \multicolumn{1}{|c|}{and char. ext.} & & & & \\
    & & \multicolumn{1}{|c|}{Theorem~\ref{thm:tourn}} & & & & \\ \cline{3-3}\noalign{\smallskip}\noalign{\smallskip}\noalign{\smallskip}
   \end{tabular}
   \caption{A summary of results. For each combination of excluded configurations, we indicate the value of the extremum (the largest size of a triple system on $n$ vertices obeying those exclusions) to the extent that it is known. We also point out the existence of a characterization of the extremal constructions when available. For earlier results we cite the papers, for new ones we refer to the relevant theorems in the present paper.}
\end{table}
\end{center}


\section{The starting case: no triangle at all}
\message{The starting case: no triangle at all}
\label{sec:all}

How large can a triple system $\FF$ on $n$
 vertices be, if it contains no triangle at all? Here is a construction originally suggested by
 Andr\'as Gy\'arf\'as as a conjectured optimum. Let $n \ge 3$ be given, and let $k$ be a positive integer such that $2k < n$. Partition the set $V$ of $n$ vertices into $k$ pairs $A_1,\ldots,A_k$ and a set $B$ of size $n-2k$. Let $\FF_{n,k}$ consist of all $k(n-2k)$ triples of the form $A_i \cup \{b\}$ where $1 \le i \le k$ and $b \in B$. Then $\FF_{n,k}$ contains no triangles. To maximize the size $k(n-2k)$ for a given $n$, we choose $k$ to be as close as possible to $\frac{n}{4}$ (this choice is unique unless $n \equiv 2\, (\mathrm{mod}\, 4)$, in which case the two choices $k = \frac{n \pm 2}{4}$ are equally good). With this optimal choice of $k$ we have $|\FF_{n,k}| = \lfloor \frac{1}{8}n^2 \rfloor$.

 The optimality of this construction among all triangle-free triple systems was 
 proved by Gy\H{o}ri~\cite{GyE}. 
Here we reproduce the proof and extend the analysis to obtain also a characterization of the extremal constructions.

\begin{thm} \label{thm:all}
$$\ex(n, \AAA\BB\CC\DD)=\lfloor \frac{1}{8}n^2 \rfloor$$
and equality is attained by a triple system $\FF$ if and only if it is of the form $\FF_{n,k}$ with $k$ chosen optimally as indicated above.
\end{thm}

\begin{proof}  
Let ${\FF}$ be a triple system not containing any of $\AAA, \BB, \CC, \DD.$
First observe that not containing $\CC$ and $\DD$ implies (in fact, is
equivalent to) that every triple contains at least two pairs which are its own.
For those triples containing three own pairs choose two that we single out as
special. Consider the graph $G$ on $V({\FF})$ formed by the own pairs and
the special own pairs of the triples containing two or three own pairs,
respectively. We claim that $G$ is triangle-free. Indeed, if a triangle were
formed by own pairs of three different triples, then we would have a triangle
in ${\FF}$ in the hypergraph sense, that is, one of our forbidden
configurations would appear. Thus at least two edges of a triangle in $G$
should be own pairs of the same triple $F$. However, in that case the third
edge of the triangle is also contained in $F$, thus it cannot be the own pair
of another triple $F'$. But it also cannot be an edge of $G$ as a third own
pair of $F$ since every triple in ${\FF}$ is represented by exactly two pairs
in $G$. Therefore a triangle in $G$ cannot occur.
This implies by Tur\'an's theorem that $|E(G)|\leq \frac{1}{4}n^2$. Since
$|E(G)|=2|\FF|$, we got $|\FF|\leq \lfloor \frac{1}{8}n^2 \rfloor.$

It remains to show that if $|\FF| = \lfloor \frac{1}{8}n^2 \rfloor$ then $\FF$ must be of the form $\FF_{n,k}$. Note that $|E(G)| = 2 \lfloor \frac{1}{8}n^2 \rfloor$, and this number equals $\lfloor \frac{1}{4}n^2 \rfloor$ unless $n \equiv 2\,(\mathrm{mod}\, 4)$, in which case it equals $\lfloor \frac{1}{4}n^2 \rfloor - 1$. By known characterizations
(see~\cite{ErdRadem})
 of the extremal and next-to-extremal cases for Tur\'an's theorem, $G$ must be bipartite with parts $A$ and $B$ satisfying one of the following:
\begin{itemize}
    \item[(a)] $n \not\equiv 2\,(\mathrm{mod}\, 4)$, $|A|$ and $|B|$ differ by at most $1$ and $G$ is complete bipartite.
    \item[(b)] $n \equiv 2\,(\mathrm{mod}\, 4)$, $|A|=|B|$ and $G$ is complete bipartite minus one edge.
    \item[(c)] $n \equiv 2\,(\mathrm{mod}\, 4)$, $|A|$ and $|B|$ differ by $2$ and $G$ is complete bipartite.
\end{itemize}

Each triple in $\FF$ contributes two own pairs to $G$, hence contains two vertices in $A$ and one in $B$ or vice versa. We claim that all triples in $\FF$ must be of the same kind, i.e., the same part contains two vertices from each triple. Assume, for the sake of contradiction, that there are triples $F_A, F_B \in \FF$ such that $F_A \cap A = \{x,x'\}$ and $F_B \cap B = \{y,y'\}$. Out of the four pairs in $\{x,x'\} \times \{y,y'\}$, at most one is a non-edge of $G$. Hence w.l.o.g.\ $x'y, yx, xy'$ form a path in $G$. Say these edges of $G$ represent the triples $F_1,F_2,F_3$ in $\FF$, respectively. Clearly, $F_1,F_2$ and $F_3$ cannot all be the same triple, so w.l.o.g.\ $F_1 \neq F_2$, and neither $F_1$ nor $F_2$ contains $\{x,x'\}$. But then $F_A,F_1,F_2$ are three distinct triples in $\FF$ which form a forbidden triangle.

Thus, we assume w.l.o.g.\ that every triple in $\FF$ contains two vertices in $A$ and one in $B$. For each vertex $b \in B$, its link in $\FF$ consists of disjoint pairs in $A$ (otherwise, we would have an edge of $G$ that belongs to two triples in $\FF$). This gives the upper bound \[ |\FF| \le \lfloor \frac{|A|}{2} \rfloor \cdot |B|.\]
If $|A|$ is odd, this bound is strictly smaller than $\lfloor \frac{1}{8}n^2 \rfloor$. Thus $|A|$ must be even, ruling out case (b) above. Moreover, the bound must hold as an equality, so the link of each $b \in B$ is a partition of $A$ into $k$ pairs $A_1,\ldots,A_k$, where $k = \frac{|A|}{2}$. Furthermore, this partition must be the same for all $b \in B$, otherwise there would be a forbidden triangle. This shows that $\FF$ is of the form $\FF_{n,k}$, and the conditions on the difference between $|A|$ and $|B|$ in cases (a), (c) above imply the optimal choice of $k$.
\end{proof}

\section{Earlier results}
\message{Earlier results}

The four cases where we exclude just one type of 
 triangle creating
 configurations $\AAA,\BB,\CC$, or $\DD$, were already investigated
 in some different context.
In particular, $\ex(n,\AAA)={{n-1}\choose 2}$ for large enough $n$ is proven in~\cite{FF87}
 (as Theorem 3.3). This result also determines that all extremal
constructions are {\em full stars},  $3$-uniform hypergraphs containing all edges through a fixed vertex.
\begin{equation*}\label{eq:AB}
  \ex(n,\AAA)={n-1\choose 2} \text{ for } n \ge n_0.
  \end{equation*}
Observing that this construction also does not contain
configuration $\BB$, we have that the same holds for
 $\AAA \BB$-free hypergraphs, too.
Here (in Section~\ref{sec:AB}) we prove that $n_0=5$ for $\ex(n,\AAA\BB)$.


An early attempt to generalize Tur\'an's theorem to hypergraphs
 (for triple systems) was proposed by Erd\H os and
 Katona, who conjectured that the largest triple system on
 $n$ vertices such that no three distinct triples
 $A,B,C$ satisfy $A\bigtriangleup B \subset C$ is the complete $3$-partite hypergraph with part sizes as equal as possible.
Here $\bigtriangleup$ is the symmetric difference,
 $A\bigtriangleup B=(A\setminus B)\cup (B\setminus A)$.
One can see that (for $A\neq B$) $A\bigtriangleup B \subset C$
 can only occur when $A$ and $B$ have two common vertices and
 the three triples form either a configuration $\BB$ or $\CC$.
The conjecture was verified by Bollob\'as~\cite{BB} thus showing
\begin{equation*}\label{eq:BC}
 \ex(n,\BB\CC)=\lfloor{n\over 3}\rfloor
  \lfloor{{n+1}\over 3}\rfloor \lfloor{{n+2}\over 3}\rfloor .
 \end{equation*}
Frankl and F\"uredi (in \cite{FF83} as Theorem~4) sharpened
 Bollob\'as's theorem
 proving that for $n\geq 3000$ it is sufficient to exclude only
 $\BB$.
Subsequently, a new proof was given by Keevash and Mubayi~\cite{KeeMub}
 who also significantly lowered the constraint for $n$:
\begin{equation*}\label{eq:B}
 \ex(n,\BB)=\lfloor{n\over 3}\rfloor
  \lfloor{{n+1}\over 3}\rfloor \lfloor{{n+2}\over 3}\rfloor
   \quad {\rm for\,\, }n\geq 33,
 \end{equation*}
 with equality only for the balanced complete 3-partite hypergraph.
They also proved that the extremal family is rather stable, namely
 for any $\varepsilon > 0$ there exists  $\delta>0$
 such that if $\HH$ is a $\BB$-free triple system on $n$ vertices
 with at least $(1-\delta)\frac{1}{27}n^3$ edges, then there exists a
 partition of the vertex set of $\mathcal H$ as
 $V({\mathcal H})=V_1\cup V_2\cup V_3$ so that all but at most
 $\varepsilon n^3$ edges of $\mathcal H$ have one vertex in each $V_i$.

The value of $\ex(n,\CC)$ is not known, it is a famous problem
 proposed by Erd\H os and S\'os~\cite{EPSos} that is thought to be
 difficult. Frankl and F\"uredi~\cite{FF84} gave a recursive
 construction based on the design $S_2(6,3,2)$, yielding the best known lower bound $\ex(n,\CC) \ge \frac{2}{7} {n \choose 3} \bigl(1-o(1)\bigr)$. An upper bound with $\frac{1}{3}$ instead of $\frac{2}{7}$ was proved by de~Caen~\cite{deCa} and Sidorenko~\cite{Sido}. A sequence of small improvements of the coefficient $\frac{1}{3}$ followed. The current record, obtained by Falgas-Ravry and Vaughan~\cite{FRV13} using semi-definite programming, is slightly less than $0.287$ (close to but still larger than $\frac{2}{7}$). Thus
\begin{equation*}\label{eq:C}
   \frac{2}{7} {n \choose 3} \bigl( 1-o(1) \bigr) \leq \ex(n,\CC)
      \leq 0.287 {n \choose 3} \bigl( 1+o(1) \bigr).
  \end{equation*}

Configuration $\DD$ is the most restrictive in terms of the largest size of a triple system excluding it: the upper bound $\ex(n,\DD) \le \frac{n(n-1)}{3}$ follows from Theorem~3.8 of Frankl and F\"uredi~\cite{FF87}.
Determining the exact value of $\ex(n,\DD)$ takes more work, which we carry out below in Section~\ref{sec:DBD}.

\section{Excluding $\AAA$ and $\BB$} \label{sec:AB}

Concerning  $\ex(n,\AAA \BB)$ we prove a stronger statement.

Define configuration $\AAA^+$ as a six vertex triple system of four edges
 with $V=\{1,\dots, 6\}, E=\{ 123, 124,135,236\}.$
This is the same as $\AAA$ with the  middle triangle $123$ as an additional triple.
Similarly, configuration $\BB^+$  on five vertices
 is the same as $\BB$ with one more central triple $123$, i.e.,
 $V=\{1,\dots, 5\}, E=\{ 123, 125,134,234\}.$
We obviously have $\ex(n,\AAA \BB)\le\ex(n,\AAA^+\BB^+)$.  Since the full star contains neither configuration $\AAA$ nor configuration $\BB$, the following theorem implies the similar statement for $\ex(n,\AAA \BB)$ in place of $\ex(n,\AAA^+\BB^+)$.

\begin{thm} \label{thm:AB}
\begin{equation*}
 \ex(n,\AAA^+\BB^+)={n-1\choose 2} \text{ for } n \ge 5
\end{equation*}
and the only extremal family is the full star.
\end{thm}

\begin{proof}
Let $\FF\subseteq {[n]\choose 3}$ be an optimal $\AAA^+\BB^+ $-free family.
Define $\FF_1\subseteq \FF$ as the family of triples with an own pair,
$\FF_1:=\left\{ xyz\in \FF: \min\{ \deg_{\FF}(xy), \deg_{\FF}(yz), \deg_{\FF}(zx) \}=1   \right\}$, and let $\FF_2:= \FF\setminus \FF_1$.
Our main observation is that for every $F=xyz\in \FF_2$ we have $\deg_{\FF}(xy)=\deg_{\FF}(yz)=\deg_{\FF}(zx)=2$. Indeed,
by definition, all these three pairs have degrees  at least $2$, so one can find triples
$xya$, $yzb$, $zxc\in \FF$ such that $\{ x,y,z\} \cap \{ a,b,c\}= \emptyset$.
These three triples together with $xyz$ form a configuration $\AAA^+$ or $\BB^+$ unless $a=b=c$. We obtain that $a$ is unique so $\deg_\FF(xy)=2$, and similarly for $b$ and $c$.

Define a nonnegative function $w(e,F):{[n]\choose 2}\times {[n]\choose 3}\to \mathbb R$ as follows.
It is always $0$ except when $e\subset F$, $F\in \FF$, when it is $\frac{1}{\deg_\FF(e)}$.
Note that here $\frac{1}{\deg_\FF(e)}\geq \frac{1}{n-2}$.
For every pair $e\in \partial \FF$ we have
 $$
   w(e):=\sum_{F\in \FF} w(e,F)= \sum_{F: e\subset F\in \FF} \frac{1}{\deg_{\FF}(e)}=1.
   $$
For every $F\in \FF_1$, $F=xyz$ we have
\begin{equation*}\label{eq:41}
   w(F):=\sum_{e\subset F, |e|=2}  w(e,F)=  \frac{1}{\deg_\FF (xy)} +\frac{1}{\deg_\FF (yz)} +\frac{1}{\deg_\FF (zx)}
      \geq 1+ \frac{2}{n-2}= \frac{n}{n-2}.
  \end{equation*}
For each $F\in \FF_2$ we  have $w(F)=\frac{3}{2}$. Assuming that $n \ge 6$, this is at least $\frac{n}{n-2}$, hence
\begin{equation*}\label{eq:42}
  {n\choose 2}\geq |\partial \FF|= \sum_{e \in \partial \FF} w(e) = \sum_e \sum_{F\in \FF} w(e,F)=
      \sum_{F\in \FF} w(F)\geq \frac{n}{n-2}|\FF|.
  \end{equation*}
This yields $ {n-1\choose 2}\geq |\FF|$. For $n\geq 7$ we infer that equality can hold only if $\FF=\FF_1$, and each $F\in \FF$ has a unique own $2$-subset.
The set-pair theorem of Bollob\'as~\cite{Boll65} gives that $\FF$ is a full star.

In case of $\FF_2=\emptyset$ the above proof works for all $n\geq 5$.
If $n=6$ and $\FF_2\neq \emptyset$ then it contains a $\binom{[4]}{3}$. The additional triples must contain $\{ 5,6\}$, hence there are at most four more of them ($x56$ with $1\leq x\leq 4$). We get $|\FF|\leq {4\choose 3}+4< {5\choose 2}$.
Finally, in the case  $n=5$ and $\FF_2\neq \emptyset$ the family $\FF$ consists of the four triples of $\binom{[4]}{3}$, hence  $|\FF|< {4\choose 2}$.
\end{proof}

\section{Excluding $\DD$ (or $\BB$ and $\DD$)} \label{sec:DBD}
As mentioned, the upper bound $\ex(n,\DD) \le \frac{n(n-1)}{3}$ follows from~\cite{FF87}. Here we conduct a careful analysis showing that this bound is either tight or nearly tight, depending on the residue of $n$ modulo $12$. The same analysis also determines $\ex(n,\BB\DD)$. The following theorem summarizes our findings.

\begin{thm} \label{thm:DBD}
There exists $n_0$ such that for all $n \ge n_0$ we have:
\begin{equation*}
\ex(n,\DD)=\ex(n,\BB\DD)=
    \begin{cases}
        \dfrac{n(n-1)}{3} & \text{if } n \equiv 1,4\, (\mathrm{mod}\, 12)\\
        \dfrac{n(n-1)}{3} - 4 & \text{if } n \equiv 7,10\, (\mathrm{mod}\, 12)\\
        \dfrac{n(n-2)}{3} & \text{if } n \equiv 0,2,3,8\, (\mathrm{mod}\, 12)\\
        \dfrac{n(n-2)}{3} - 1 & \text{if } n \equiv 5,11\, (\mathrm{mod}\, 12)
    \end{cases}
\end{equation*}
\[ \ex(n,\DD) = \frac{n(n-2)}{3} \text{ and } \ex(n,\BB\DD) = \frac{n(n-2)}{3} - 1 \text{ if } n \equiv 6,9\, (\mathrm{mod}\, 12).\] Without the requirement $n \ge n_0$, the indicated quantities are always valid as upper bounds, and as exact values for $n \equiv 0,1,3,4\, (\mathrm{mod}\, 12)$.
\end{thm}

\subsection{Preparations for the proof}\label{ss:51}
In preparation for the proof, we develop or recall some tools that will be used in it. Let $\FF$ be a triple system on a vertex set $V$, $|V|=n$. We call two triples in $\FF$ \emph{contiguous} 
 if they share two vertices. The contiguity relation defines a graph on $\FF$, which allows us to partition $\FF$ uniquely into its connected components with respect to contiguity: \[ \FF = \bigcup_{i=1}^r \FF_i \] Let us denote by $\partial \FF_i$ the shadow of $\FF_i$, i.e., the set of those pairs contained in some triple in $\FF_i$. Viewing the $\partial \FF_i$ as graphs, they are edge-disjoint.

Two special types of $\FF_i$ serve as building blocks of $\DD$-free triple systems. One, that we call a \emph{cluster}, consists of $3$ or $4$ triples on $4$ vertices; its shadow is a $K_4$ graph. The other, that we call a $k$-\emph{crown}, consists of some number $k \ge 1$ of triples, all sharing a fixed pair of vertices $x$ and $y$; its shadow is a book graph $B_k$ on $k+2$ vertices, consisting of $k$ triangles sharing an edge. In the case $k=1$ we get a single triple (triangle). The following easy-to-check fact is recorded here for later reference.

\begin{obs} \label{obs:CC}
A triple system $\FF$ is $\DD$-free if and only if each of its connected components $\FF_i$ is either a cluster or a $k$-crown for some $k \ge 1$.
\end{obs}

Thus, constructing a $\DD$-free triple system $\FF$ on $V$ amounts to packing edge-disjoint copies of $K_4$ and some $B_k$'s in the complete graph $K_n$. As we are interested in the largest possible $\FF$, we may assume that each copy of $K_4$ contributes $4$ triples to $\FF$. The ratio $\frac{|\FF_i|}{|E(\partial \FF_i)|}$ is $\frac{2}{3}$ in each cluster and less than $\frac{1}{2}$ in each crown. Hence we would like to pack in $K_n$ as many copies of $K_4$ as we can, and use the leftover edges, if any, to add copies of some $B_k$'s if possible. The maximum number of copies of a fixed graph $H$ that can be packed in $K_n$ was determined for large enough $n$ by Caro and Yuster~\cite{CY}. Their main tool, which will also be ours, was the following decomposition theorem of Gustavsson, which extends Wilson's theorem (see Theorem~\ref{thm:W} below) from complete graphs to sufficiently dense ones.

\begin{thm}[Gustavsson~\cite{Gu}, Glock et al.~\cite{GKLMO}] \label{thm:Gu}
For every fixed graph $H$ there exist $n_0 = n_0(H)$ and $\varepsilon = \varepsilon(H) > 0$ such that  the following holds for all $n \ge n_0$ and every graph $G$ on $n$ vertices with all vertex degrees greater than $(1 - \varepsilon)n$:\\ There exists a decomposition of the edge set $E(G)$ into copies of $H$ if and only if the number of edges of $H$ divides that of $G$, and the greatest common divisor of the vertex degrees in $H$ divides that of the vertex degrees in $G$.
\end{thm}

We will apply this theorem to $H=K_4$. In this case, the divisibility conditions become: $6\vert |E(G)|$ and $3\vert \deg_G(v)$ for all $v \in V$. These conditions are satisfied by $G=K_n$ if and only if $n \equiv 1,4\, (\mathrm{mod}\, 12)$. For all (not just large enough) $n \equiv 1,4\, (\mathrm{mod}\, 12)$, a full packing of copies of $K_4$ in $K_n$ is known to exist (Hanani~\cite{Hanani}). When such a full packing does not exist, we will use a decomposition into $K_4$'s of a large subgraph $G$ of $K_n$, specified by describing its complement $\overline{G}$. In terms of $\overline{G}$ the divisibility conditions become:
\[ |E(\overline{G})| \equiv {n \choose 2}\, (\mathrm{mod}\, 6) \text{ and } \deg_{\overline{G}}(v) \equiv n-1\, (\mathrm{mod}\, 3) \text{ for all } v \in V.\] The graphs $\overline{G}$ that we will specify have constant maximum degree, which ensures that all vertex degrees in $G$ will be greater than $(1 - \varepsilon)n$ for large enough $n$. The $K_4$-decomposition of $G$ provided by Theorem~\ref{thm:Gu}, plus some edge-disjoint $B_k$'s in $\overline{G}$ when available, will naturally define a large $\DD$-free triple system $\FF$. In most cases this $\FF$ will also be $\BB$-free; when this is not the case we will modify $\overline{G}$ to get a slightly smaller triple system which is $\BB\DD$-free.

Another useful way to analyze triple systems $\FF$ is by describing their links. Recall that the link $\FF[x]$ of a vertex $x$ is the graph on $V \setminus \{x\}$ whose edges are those pairs of vertices that form with $x$ a triple in $\FF$. Here is another simple characterization of $\DD$-free triple systems.

\begin{obs} \label{obs:Lk}
A triple system $\FF$ is $\DD$-free if and only if each of its links $\FF[x]$ is a vertex-disjoint union of triangles and stars (including the cases of an isolated vertex or edge). For such $\FF$, denoting by $\mathrm{st}(\FF[x])$ the number of stars in $\FF[x]$, we have \[ |\FF| = \frac{1}{3} \sum_{x \in V} \bigl( n-1-\mathrm{st}(\FF[x]) \bigr).\]
\end{obs}

Note that this immediately implies the following upper bound:
\begin{equation*}
    \ex(n,\DD) \le
    \begin{cases}
        \dfrac{n(n-1)}{3} & \text{if } n \equiv 1\, (\mathrm{mod}\, 3)\\
        \dfrac{n(n-2)}{3} & \text{if } n\not\equiv 1\, (\mathrm{mod}\, 3)
    \end{cases}
\end{equation*}

In our detailed analysis, depending on the residue of $n$ modulo $12$, we will construct $\DD$ (and even $\BB\DD$)-free triple systems that match this upper bound exactly or up to a small constant difference, and prove their optimality.

\subsection{Proof of Theorem~\ref{thm:DBD}}

\noindent \textbf{Case 1} $n \equiv 1,4\, (\mathrm{mod}\, 12)$

We saw that $\frac{n(n-1)}{3}$ is an upper bound on the size of a $\DD$-free triple system. A matching construction is obtained by placing $4$ triples in each copy of $K_4$ in a full packing in $K_n$, which exists by~\cite{Hanani}. This gives a triple sytem $\FF$ of size $|\FF| = 4 \cdot \frac{1}{6}{n \choose 2} = \frac{n(n-1)}{3}$, which is $\BB$-free as well.

\vspace{10pt}

\noindent \textbf{Case 2} $n \equiv 0,3\, (\mathrm{mod}\, 12)$

We saw that $\frac{n(n-2)}{3}$ is an upper bound on the size of a $\DD$-free triple system in this case. For a matching construction, we use a packing of $\frac{n(n-3)}{12}$ copies of $K_4$ and $\frac{n}{3}$ triangles in $K_n$. Such a packing can be obtained by starting with a full packing of copies of $K_4$ in $K_{n+1}$, which exists by~\cite{Hanani}, and deleting one vertex. Placing $4$ triples in each copy of $K_4$ and one in each triangle gives a $\BB\DD$-free triple system $\FF$ of size $|\FF| = \frac{n(n-3)}{3} + \frac{n}{3} = \frac{n(n-2)}{3}$.

\vspace{10pt}

\noindent \textbf{Case 3} $n \equiv 2,8\, (\mathrm{mod}\, 12)$

Again, $\frac{n(n-2)}{3}$ is an upper bound. For a matching construction for large enough $n$, we use Theorem~\ref{thm:Gu} as follows. We take $\overline{G}$ to be a perfect matching in $K_n$, noting that it satisfies the required conditions: $|E(\overline{G})| = \frac{n}{2} \equiv {n \choose 2}\, (\mathrm{mod}\, 6)$ and $\deg_{\overline{G}}(v) = 1 \equiv n-1\, (\mathrm{mod}\, 3)$. Placing $4$ triples in each copy of $K_4$ in a $K_4$-decomposition of $G$ we get a $\BB\DD$-free triple system $\FF$ of size $|\FF| = \frac{2}{3}\bigl( {n \choose 2} - \frac{n}{2} \bigr) = \frac{n(n-2)}{3}$.

\vspace{10pt}

\noindent \textbf{Case 4} $n \equiv 5,11\, (\mathrm{mod}\, 12)$

We start with the construction for large enough $n$, using Theorem~\ref{thm:Gu}. We take $\overline{G}$ to consist of a matching of $\frac{n-5}{2}$ edges and a star on the remaining $5$ vertices. The required conditions are satisfied: $|E(\overline{G})| = \frac{n-5}{2} + 4 \equiv {n \choose 2}\, (\mathrm{mod}\, 6)$ and $\deg_{\overline{G}}(v) \in \{1,4\} \equiv n-1\, (\mathrm{mod}\, 3)$ for all $v \in V$. Placing $4$ triples in each copy of $K_4$ in a decomposition of $G$, we get a $\BB\DD$-free triple system $\FF$ of size $|\FF| = \frac{2}{3} \bigl( {n \choose 2} - \frac{n+3}{2} \bigr) = \frac{n(n-2)}{3} - 1$.

To see that this is optimal, assume that there is a $\DD$-free triple system $\FF$ of size $|\FF| = \frac{n(n-2)}{3}$. By Observation~\ref{obs:Lk}, we must have $\mathrm{st}(\FF[x]) = 1$ for every vertex $x$. Consider the mapping defined on $V$ by $x \mapsto y$, where $y \ne x$ is the center of the star in the link $\FF[x]$. This is well defined, because there is a unique star and it has a unique center (the star cannot be a single edge because it has $1\, (\mathrm{mod}\, 3)$ vertices). Note also that $x \mapsto y$ implies $y \mapsto x$, so this mapping partitions $V$ into pairs, which is impossible because $n$ is odd.

\vspace{10pt}

\noindent \textbf{Case 5a} $n \equiv 6,9\, (\mathrm{mod}\, 12)$, excluding $\DD$

We saw that  $\frac{n(n-2)}{3}$ is an upper bound in this case. For a matching construction for large enough $n$, we use Theorem~\ref{thm:Gu} as follows. We take $\overline{G}$ to consist of $\frac{n-6}{3}$ vertex-disjoint triangles and a $B_4$ on the remaining vertices. The required conditions are satisfied: $|E(\overline{G})| = (n-6) + 9 \equiv {n\choose 2}\, (\mathrm{mod}\, 6)$ and $\deg_{\overline{G}}(v) \in \{2,5\} \equiv n-1\, (\mathrm{mod}\, 3)$ for all $v \in V$. Placing $4$ triples in each copy of $K_4$ in a decomposition of $G$, one in each triangle in $\overline{G}$, and $4$ in the book $B_4$, we get a $\DD$-free triple system $\FF$ of size $|\FF| = \frac{2}{3}\bigl( {n \choose 2} - n - 3 \bigr) + \frac{n-6}{3} + 4 = \frac{n(n-2)}{3}$.

\vspace{10pt}

\noindent \textbf{Case 5b} $n \equiv 6,9\, (\mathrm{mod}\, 12)$, excluding $\BB\DD$

Note that the construction in Case 5a is not $\BB$-free, because two triples $xyz, xyw$ from the crown on $B_4$ form a $\BB$ configuration with a triple containing $\{z,w\}$. We modify the construction, letting $\overline{G}$ consist of $\frac{n-9}{3}$ vertex-disjoint triangles and $4$ more triangles on the remaining vertices, all sharing one vertex. The modification does not change the number of edges in $\overline{G}$, and the only new degree is $8$, hence the required conditions are satisfied. Placing $4$ triples in each copy of $K_4$ in a decomposition of $G$ and one in each triangle in $\overline{G}$, we get a $\BB\DD$-free triple system $\FF$ of size $|\FF| = \frac{2}{3}\bigl( {n \choose 2} - n - 3 \bigr) + \frac{n-9}{3} + 4 = \frac{n(n-2)}{3} - 1$.

To see that this is optimal, assume that there is a $\BB\DD$-free triple system $\FF$ of size $|\FF| = \frac{n(n-2)}{3}$. Then in each link $\FF[x]$ there must be exactly one star. We claim that none of the $\FF_i$ (the connected components of $\FF$ with respect to contiguity, as defined at the beginning of Subsection~\ref{ss:51}) can be a $k$-crown with $k \ge 2$. Suppose that $\FF_i$ is such a $k$-crown, consisting of the triples $xyz_i$, $i = 1,\ldots,k$. Then the link $\FF[z_1]$ contains $xy$ as an isolated edge, and also $z_2,\ldots,z_k$ as isolated vertices (the latter is due to $\BB$-freeness). This contradicts the uniqueness of the star in each link. By Observation~\ref{obs:CC}, it follows that each $\FF_i$ is either a cluster or a $1$-crown (i.e., a triple contiguous to no other triple). This implies that each link is a vertex-disjoint union of $\frac{n-3}{3}$ triangles and an isolated edge. But this is impossible: if every vertex appears in $\frac{n-3}{3}$ clusters, then the number of clusters is $\frac{1}{4}n \cdot \frac{n-3}{3}$, which is not an integer.

\vspace{10pt}

\noindent \textbf{Case 6} $n \equiv 7,10\, (\mathrm{mod}\, 12)$

We start with the construction for large enough $n$, using Theorem~\ref{thm:Gu}. We take $\overline{G}$ to be the triangular prism graph (two vertex-disjoint triangles with a perfect matching of their vertices) plus $n-6$ isolated vertices. The required conditions are satisfied: $|E(\overline{G})| = 9 \equiv {n \choose 2}\, (\mathrm{mod}\, 6)$ and $\deg_{\overline{G}}(v) \in \{0,3\} \equiv n-1\, (\mathrm{mod}\, 3)$ for all $v \in V$. Placing $4$ triples in each copy of $K_4$ in a decomposition of $G$ and one in each of the two triangles in $\overline{G}$, we get a $\BB\DD$-free triple system $\FF$ of size $|\FF| = \frac{2}{3} \bigl( {n \choose 2} - 9 \bigr) + 2 = \frac{n(n-1)}{3} - 4$.

To see that this is optimal, assume that there is a $\DD$-free triple system $\FF$ of size $|\FF| = \frac{n(n-1)}{3} - 3$. Consider the decomposition of $\FF$ into its connected components $\FF_i$ with respect to contiguity. Let the graph $G$ be the edge-disjoint union of the $K_4$'s that are the shadows of the cluster $\FF_i$'s. Let $\overline{G}$ be the complement of $G$. The divisibility conditions are: \[ |E(\overline{G})| \equiv 3\, (\mathrm{mod}\, 6) \text{ and } \deg_{\overline{G}}(v) \equiv 0\, (\mathrm{mod}\, 3) \text { for all } v \in V. \] The shadows of the crown $\FF_i$'s are edge-disjoint subgraphs of $\overline{G}$, and the number of triples that each of them carries is less than half its number of edges. Thus we have \[ \frac{n(n-1)}{3} - 3 = |\FF| \le \frac{2}{3}|E(G)| + \frac{1}{2}|E(\overline{G})| = \frac{2}{3}{n \choose 2} - \frac{1}{6}|E(\overline{G})| \] from which we deduce that $|E(\overline{G})| \le 18$. Together with the fact that $|E(\overline{G})| \equiv 3\, (\mathrm{mod}\, 6)$, this leaves only three possibilities for $|E(\overline{G})|$, that we can rule out one-by-one:
\begin{itemize}
    \item $|E(\overline{G})| = 3$\\By $\deg_{\overline{G}}(v) \equiv 0\, (\mathrm{mod}\, 3)$, all non-zero degrees must be at least $3$, which is impossible with $3$ edges.
    \item $|E(\overline{G})| = 9$\\All non-zero degrees must be exactly $3$, because there are not enough edges to allow a degree of $6$ or more. Ignoring isolated vertices, $\overline{G}$ is a $3$-regular graph on $6$ vertices. A quick case check shows that such a graph contains at most two triangles. Hence the total number of triples in $\FF$ is at most the number in our construction, namely $\frac{n(n-1)}{3} - 4$.
    \item $|E(\overline{G})| = 15$\\Let $t$ be the number of triples in $\FF$ coming from $\overline{G}$. From $\frac{n(n-1)}{3} - 3 = |\FF| \le \frac{2}{3} \bigl( {n \choose 2} - 15 \bigr) + t$ we get that $t \ge 7$. The only way to obtain $7$ triples from crowns whose edge-disjoint shadows have at most $15$ edges in total is to use one $7$-crown. But then $\overline{G}$ is a book graph $B_7$ (plus isolated vertices), having vertex degrees $2$ and $8$ which are not $0\, (\mathrm{mod}\, 3)$.
\end{itemize}

This completes the proof of Theorem~\ref{thm:DBD}. \qed

\section{Excluding $\CC$ and $\DD$ (and $\AAA$, optionally)}
\message{Excluding $CC$ and $DD$}

\begin{obs}\label{obs:CD}
The triple system $\FF$ avoids  $\CC$ and $\DD$ if and only if each triple $F\in \FF$ contains at least two own pairs.
  \end{obs}
Indeed, if  $123\in \FF$ with $\deg_\FF (12), \deg_\FF (13)\ge 2$, then the three triples
$ 123, 12x, 13y $ either form a $\CC$ (in case of $x=y$), or a $\DD$ (in case of $x\neq y$).
On the other hand, both $\CC$ and $\DD$ have a member with at most one own pair.

This observation gives
$\ex(n,\CC\DD)\leq {1\over 2}{n\choose 2}=\frac{1}{4}n^2 \bigl(1-o(1)\bigr)$
and a construction of matching order of magnitude is not hard to find.
In this case, however, we will give the exact value of $\ex(n,\CC\DD)$
together with a characterization of all the extremal constructions. 
As it turns out it will be just one little step further to give a
similar characterization of all extremal $\AAA\CC\DD$-free constructions.  

\subsection{A construction based on a tournament}

First we describe a construction we will call the {\it tournament
construction}. Consider a maximum size matching in the graph $K_n$. Fix
a tournament $T$ on a vertex set consisting of the edges of this
matching. (Thus our tournament is on $\lfloor{n\over 2}\rfloor$
vertices.) Let $\FF_T$ be the $3$-uniform hypergraph the edges of which
are the triples consisting of a matching edge plus a vertex of another matching
edge into which the previous matching edge (as a vertex of $T$) sends an edge
in $T$. In case $n$ is odd, the third vertex of an edge in $\FF_T$ can
also be the unique unmatched vertex of our original $K_n$. The number of triples
in $\FF_T$ is $2{{n\over 2}\choose 2}$ if $n$ is even and
$2{{{n-1}\over 2}\choose 2}+{{n-1}\over 2}$ if $n$ is odd. (Both of the latter
can be written as $\lfloor \frac{1}{4}(n-1)^2 \rfloor.$) It is easy to check
that each triple in $\FF_T$ (for any possible $T$) contains two own
pairs, thus $\CC$ and $\DD$ do not appear in $\FF_T$. This gives
$\ex(n,\CC\DD)\ge \lfloor \frac{1}{4}(n-1)^2 \rfloor.$

\begin{thm} \label{thm:tourn}
$$\ex(n,\CC\DD)=\lfloor \frac{1}{4}(n-1)^2 \rfloor$$
and equality is attained by a triple system $\FF$ if and only if it is
of the form $\FF_T$ for some tournament
$T$.
\end{thm}

Before the proof we obtain the following immediate consequence.

\begin{cor} \label{cor:ACD}
$$\ex(n,\AAA\CC\DD)=\lfloor \frac{1}{4}(n-1)^2 \rfloor$$
and equality is attained by a triple system $\FF$ if and only if it is
of the form $\FF_T$ for the transitive tournament $T$.
\end{cor}

\noindent{\bf Proof of Corollary~\ref{cor:ACD}.}\quad
If $T$ is not transitive then it contains a cyclically oriented triangle which
gives rise to a configuration $\AAA$ in $\FF_T$.
On the other hand, if $T$
is transitive then $\AAA$ cannot occur in $\FF_T$.
This observation and the trivial inequality
$\ex(n,\AAA\CC\DD)\leq \ex(n,\CC\DD)$ imply the statement by
 Theorem~\ref{thm:tourn}.
\qed

\subsection{Proof of Theorem~\ref{thm:tourn}}

Let $\FF$ be a triple system not containing $\CC$ and $\DD$ and having the maximum
number of triples. Then $|\FF| \ge |\FF_T| = \lfloor \frac{1}{4}(n-1)^2 \rfloor$. We have already observed that each $F \in \FF$ has at least two own pairs. We form a graph $G$ on $V(\FF)$ by taking as edges exactly two own pairs from each $F \in \FF$. Let $\overline{G}$ be the complement of $G$. Note that every $F \in \FF$ has two of its pairs in $G$ and the third in $\overline{G}$.

We claim that the edges of $\overline{G}$ form a matching of size $\lfloor \frac{n}{2} \rfloor$. First, observe that $|E(G)| = 2|\FF| \ge 2\lfloor \frac{1}{4}(n-1)^2 \rfloor = \lfloor \frac{1}{2}(n-1)^2 \rfloor$. Hence $|E(\overline{G})| \le {n \choose 2} - \lfloor \frac{1}{2}(n-1)^2 \rfloor = \lfloor \frac{n}{2} \rfloor$. Now we distinguish two cases. If every vertex $x \in V(\FF)$ is incident to an edge of $\overline{G}$, then $\overline{G}$ must be a perfect matching (and $n$ is even). In the other case, consider a vertex $x$ all of whose pairs are edges of $G$. These $n-1$ pairs are own pairs of triples containing $x$, so $n$ is odd and these triples are of the form $xy_1z_1,\ldots,xy_{(n-1)/2}z_{(n-1)/2}$ where $y_1z_1,\ldots,y_{(n-1)/2}z_{(n-1)/2}$ form a perfect matching in $V(\FF) \setminus \{x\}$ and each of them is an edge of $\overline{G}$ (as the third pair of its triple). In this case, too, we get that $\overline{G}$ is a matching of size $\lfloor \frac{n}{2} \rfloor$. In particular, the inequalities above must hold as equalities, proving the first part of the theorem.

To establish the tournament structure, we take the edges of $\overline{G}$ as the maximum size matching (the vertices of the tournament). We already know that each triple in $\FF$ contains an edge of $\overline{G}$. Given any two distinct edges of $\overline{G}$, the four crossing pairs between them must be edges of $G$, hence own pairs of triples. This prevents the existence of triples going both ways between two edges of $\overline{G}$, and implies the tournament structure for $\FF$. \qed

\section{Excluding $\AAA$ and $\DD$}
For the combination $\AAA\DD$, by our previous analysis for $\AAA\CC\DD$, we can write \[\ex(n,\AAA\DD) \ge \ex(n,\AAA\CC\DD) = \lfloor \frac{1}{4}(n-1)^2 \rfloor.\]
We recall that $\ex(n,\AAA\CC\DD)$ is attained uniquely by the construction $\FF_T$ based on a transitive tournament $T$.

It turns out that the answer for $\AAA\DD$ is similar, but a bit more subtle. For any even $n \ge 4$ we can add two triples to $\FF_T$ as indicated next, and still avoid $\AAA$ and $\DD$. Say the pairs $12$ and $34$ are the elements of the tournament $T$ with the lowest and second lowest outdegree, respectively. On these four vertices, $\FF_T$ contains the two triples $134, 234$, and we can add the other two triples $123, 124$ without forming an $\AAA$ or $\DD$. We denote this augmented construction by $\FF_T^+$. For odd $n$ we cannot do better than $\frac{1}{4}(n-1)^2$, but we can modify $\FF_T$ as indicated next, and get another extremal construction (assuming $n \ge 5$). Say $0, 12, 34$ are the singleton and the two lowest outdegree elements of $T$, respectively. On these five vertices, $\FF_T$ contains the four triples $012, 034, 134, 234$. We can replace $012$ and $034$ by $123$ and $124$ without forming an $\AAA$ or $\DD$. We denote this modified construction by $\FF_T^{\pm}$.

\begin{thm} \label{thm:AD}
\begin{equation*}
\ex(n,\AAA\DD)=
    \begin{cases}
        \lfloor \frac{1}{4}(n-1)^2\rfloor + 2 & \text{if } n \ge 4 \text{ is even}\\
        \frac{1}{4}(n-1)^2 & \text{if } n \text{ is odd}
    \end{cases}
\end{equation*}
and equality is attained by a triple system $\FF$ if and only if it is of the form $\FF_T^+$ (for even $n \ge 4$) or one of the two forms $\FF_T$ or $\FF_T^{\pm}$ (for odd $n \ge 5$) for a transitive tournament $T$.
\end{thm}

\begin{proof} 
We argue by induction on $n$. The base cases $n=3,4$ are trivial, so we assume $n \ge 5$. Let $\FF$ be an $\AAA\DD$-free triple system on the vertex set $V$, with $|V|=n$ and
\begin{equation*}
|\FF| \ge f(n) \coloneqq
    \begin{cases}
        \lfloor \frac{1}{4}(n-1)^2 \rfloor + 2 & \text{if } n \text{ is even}\\
        \frac{1}{4}(n-1)^2 & \text{if } n \text{ is odd}
    \end{cases}
\end{equation*}
We have to show that $|\FF| = f(n)$ and $\FF$ is of one of the specified forms.

Assume first that $\FF$ is $\CC$-free as well. Then by Corollary~\ref{cor:ACD} we have $|\FF| \le \lfloor \frac{1}{4}(n-1)^2 \rfloor$, so this is possible only when $n$ is odd, in which case we get $|\FF| = f(n)$ and $\FF$ is of the form $\FF_T$ (again by Corollary~\ref{cor:ACD}).

From now on, we assume that $\FF$ does contain a $\CC$ configuration, say on vertices $1,2,3,4$. Note that any triple in $\FF$ that is not contained in $\{1,2,3,4\}$ intersects it in at most one vertex -- otherwise we get a $\DD$ configuration. When $n=5$ this immediately implies that $\FF$ consists only of the $4$ triples in $\{1,2,3,4\}$, which form an $\FF_T^{\pm}$. When $n=6$ this implies that $\FF$ consists of those $4$ triples and another $4$ triples of the form $i56$, $i=1,2,3,4$, yielding an $\FF_T^+$. Henceforth we assume $n \ge 7$, which will allow us to apply the induction hypothesis to the family of those triples in $\FF$ contained in $V' \coloneqq V \setminus \{1,2,3,4\}$.

For $i=1,2,3,4$, let $G_i$ be the graph whose edges are those pairs in $V'$ that form with the vertex $i$ a triple in $\FF$. By the absence of $\DD$, the graph $G_i$ has no path of three edges, implying that each connected component of $G_i$ is either a star (including the cases of an isolated vertex or edge) or a triangle. By the absence of $\AAA$, an edge of $G_i$ and an edge of $G_j$, $j \neq i$, cannot share exactly one vertex. Let $G$ be the multigraph on $V'$ formed by the $G_i$, $i=1,2,3,4$. By the two observations just made, each connected component of $G$ is either a pair of vertices with $2,3$ or $4$ parallel edges between them (say there are $m$ such pairs $P_1,\ldots,P_m$), or a star or triangle whose edges belong to a single $G_i$.

As the number of edges is at most $4$ in each $P_j$, and at most the number of vertices in every other component of $G$, we can write \[|E(G)| \le 4m + \bigl( n-4-2m \bigr) = n+2m-4.\]
On the other hand, using the induction hypothesis and the equality $f(n) - f(n-4) = 2n - 6$, we get \[f(n) \le |\FF| \le 4 + |E(G)| + f(n-4) = |E(G)| + f(n) - 2n +10,\] implying that $|E(G)| \ge 2n - 10$. Comparing these two bounds on $|E(G)|$, we immediately get $2m \ge n-6$, so $\cup_{j=1}^m P_j$ covers all but at most two vertices of $V'$. Actually, it cannot leave out two vertices, because then our upper bound on $|E(G)|$ is not tight (those two vertices carry at most one edge). It follows that $P_1,\ldots,P_m$ is a maximum matching in $V'$, i.e., $m = \lfloor \frac{n-4}{2} \rfloor$. Our bounds on $|E(G)|$ become \[2n - 10 \le |E(G)| \le 4 \lfloor \frac{n-4}{2} \rfloor.\] For odd $n$ the bounds coincide, implying that each $P_j$ has $4$ parallel edges. For even $n$ there is a slack of $2$, so there may be one exceptional $P_j$ with $2$ parallel edges, or at most two exceptional $P_j$ with $3$ parallel edges each. Note that in any case, for any two distinct $P_j,P_k$, there is a graph $G_i$ having an edge in both $P_j$ and $P_k$.

We claim that any triple $F \in \FF$, $F \subseteq V'$ must contain one of the pairs $P_j$, $j=1,\ldots,m$. If not, then $F$ is of the form $xyz$ where $x \in P_j$, $y \in P_k$, $k \neq j$, and $z$ is either in a third $P_{\ell}$ or (when $n$ is odd) is the unpaired vertex of $V'$. Let $x'$ and $y'$ be the other vertices of $P_j$ and $P_k$, respectively, and let $i \in \{1,2,3,4\}$ be such that $E(G_i)$ contains both $xx'$ and $yy'$. Then $xyz, ixx',iyy'$ give an $\AAA$ configuration, which is forbidden.

We claim further that for any two distinct pairs $P_j$ and $P_k$, there cannot coexist in $\FF$ a triple made of $P_j$ and a vertex of $P_k$ and a triple made of $P_k$ and a vertex of $P_j$. Indeed, two such triples and a triple of the form $P_j \cup \{i\}$, $i \in \{1,2,3,4\}$, would yield a $\DD$ configuration.

To summarize, the possible triples $F \in \FF$, $F \subseteq V'$ are either contained in some $P_j \cup P_k$, $j \neq k$ (and there are at most two such triples for each $j \neq k$), or consist of some $P_j$ and the unpaired vertex of $V'$ (when $n$ is odd). This gives the following upper bound on $|\FF|$:

\begin{equation*}
|\FF| \le
    \begin{cases}
        4 + 4\frac{n-4}{2} + 2{\frac{n-4}{2} \choose 2} & \text{if } n \text{ is even}\\
        4 + 4\frac{n-5}{2} + 2{\frac{n-5}{2} \choose 2} + \frac{n-5}{2} & \text{if } n \text{ is odd}
    \end{cases}
\end{equation*}
A calculation shows that for either parity, this bound equals $f(n)$. We conclude that $|\FF| = f(n)$ and all the triples counted in the upper bound are in fact in $\FF$. This $\FF$ is easily seen to be of the form $\FF_T^+$ in the even case and $\FF_T^{\pm}$ in the odd case, for a tournament $T$ in which $P_j$ beats $P_k$ if the two triples in $P_j \cup P_k$ contain $P_j$, and every $P_j$ beats $12$ and $34$. The tournament must be transitive in order to avoid an $\AAA$ configuration.
\end{proof}  

\section{Excluding $\AAA, \BB$ and $\CC$}
\subsection{A construction based on the Tur\'an graph}
Let ${\FF^*_n}$ be the following triple system on a set $V$ of $n$ vertices. Fix a vertex $x\in V$, and consider a complete bipartite graph on $V \setminus \{x\}$ with parts of size $\lfloor{{n-1}\over 2}\rfloor$ and $\lceil{{n-1}\over 2}\rceil$, respectively.
Let the triple $F$ be in $\FF^*_n$ if and only if $F$ consists of $x$ and an edge of the complete bipartite graph. It is easy to check that none of $\AAA$, $\BB$ and $\CC$ is contained in $\FF^*_n$, and its size is $\lfloor\frac{1}{4}(n-1)^2\rfloor$.

\begin{thm} \label{thm:ABC}
\begin{equation*}
\ex(n,\AAA\BB\CC)=\lfloor \frac{1}{4}(n-1)^2\rfloor
\end{equation*}
and equality is attained by a triple system $\FF$ if and only if it is isomorphic to $\FF^*_n$.
\end{thm}

\subsection{The structure of nearly extremal triangle-free graphs}
We will need the following lemma about triangle-free graphs on $n$ vertices having at least $\lfloor \frac{1}{4}(n-1)^2 \rfloor$ edges. Given a graph $G=(V,E)$, two vertices $x,y \in V$ are called \emph{remote} if the distance between them in $G$ is greater than $2$.

\begin{lem} \label{lem:ABC}
Let $G=(V,E)$ be a triangle-free graph with $|V|=n$ and $|E| \ge \lfloor \frac{1}{4}(n-1)^2 \rfloor$. Then either
\begin{enumerate}
\item $G$ is bipartite, and if $|E| > \lfloor \frac{1}{4}(n-1)^2 \rfloor$ then no two vertices of the same part are remote, or
\item $G$ is non-bipartite, and there are at most two pairs of remote vertices (none if $|E| > \lfloor \frac{1}{4}(n-1)^2 \rfloor$).
\end{enumerate}
\end{lem}

\begin{proof} 
Our proof extends a classical one by  Andrásfai, Erd\H{o}s and Gallai (see~\cite{ErdRadem}), who showed that a triangle-free graph with at least $\lfloor \frac{1}{4}(n-1)^2 \rfloor + 2$ edges must be bipartite.

Suppose first that $G$ is bipartite with parts $A$ and $B$ of cardinality $a$ and $b$, respectively, where $a+b=n$, $a \le b$. If some two vertices of the same part are remote, then at least $a$ pairs in $A \times B$ are non-edges in $G$. This implies that $|E| \le a(b-1) \le \lfloor \frac{1}{4}(n-1)^2 \rfloor$, proving the assertion in the first part of the lemma.

Next, suppose that $G$ is non-bipartite, and let $C_{\ell}$ be a shortest odd cycle in $G$, of length $\ell \ge 5$. To avoid a shorter odd cycle, $C_{\ell}$ must be induced, and every vertex in $V \setminus V(C_{\ell})$ can have at most two neighbors in $V(C_{\ell})$. Using these facts, and Tur\'an's theorem for the graph induced on $V \setminus V(C_{\ell})$, we get the upper bound \[ |E| \le \ell + 2(n - \ell) + \lfloor \frac{1}{4}(n-\ell)^2 \rfloor. \]
This bound is a decreasing function of $\ell$, and for $\ell = 7$ it is strictly less than $\lfloor \frac{1}{4}(n-1)^2 \rfloor$. Hence we must have $\ell = 5$, for which the bound becomes $|E| \le \lfloor \frac{1}{4}(n-1)^2 \rfloor + 1$.

If $|E| = \lfloor \frac{1}{4}(n-1)^2 \rfloor + 1$ then $G$ has the following structure: a $5$-cycle $C_5$, a complete bipartite graph $K_{\lfloor \frac{n-5}{2} \rfloor, \lceil \frac{n-5}{2} \rceil}$ on $V \setminus V(C_5)$, and every vertex of the latter is joined by edges to two non-adjacent vertices of the former. This graph is easily seen to have diameter $2$, i.e., no pair of remote vertices. If $|E| = \lfloor \frac{1}{4}(n-1)^2 \rfloor$ then one of the two terms $2(n-5)$ and $\lfloor \frac{1}{4}(n-5)^2 \rfloor$ in the upper bound on $|E|$ was off by $1$. If it was the former, then there is a unique vertex $x$ in $V \setminus V(C_5)$ that has a single neighbor $y$ in $V(C_5)$. Denoting the two vertices of $C_5$ at distance $2$ from $y$ by $z$ and $w$, we see that there are at most two pairs of remote vertices, $\{x,z\}$ and $\{x,w\}$. If it was the latter, then the triangle-free graph induced on $V \setminus V(C_5)$ has $\lfloor \frac{1}{4}(n-5)^2 \rfloor - 1$ edges, one less than the Tur\'an bound. This induced graph can only be of one of the following forms: the Tur\'an graph minus an edge, or (when $n-5$ is even) a $K_{\frac{n-5}{2} - 1, \frac{n-5}{2} + 1}$, or (when $n-5=5$) another copy of $C_5$. In each of these cases, one can check that there are at most two pairs of remote vertices.
\end{proof}  

\subsection{Proof of Theorem~\ref{thm:ABC}}
Let $\FF$ be an $\AAA\BB\CC$-free triple system on $n$ vertices. Then each triple in $\FF$ has at least one own pair. Choosing exactly one own pair from each triple, we get a graph $G=(V,E)$ with $|E| = |\FF|$. The absence of $\AAA\BB\CC$ in $\FF$ implies that $G$ is triangle-free, and for each triple in $\FF$, one of its pairs is an edge of $G$ and the other two pairs are remote in $G$.

We assume that $|E| = |\FF| \ge \lfloor \frac{1}{4}(n-1)^2 \rfloor$, and apply Lemma~\ref{lem:ABC} to the graph $G$. As each triple in $\FF$ contributes two pairs of remote vertices, case 2 of the lemma cannot occur. Thus we are in case 1, and $G$ is bipartite with bipartition $V=A \cup B$. Each triple in $\FF$ must contain an edge of $G$, hence two vertices in one part and one in the other. By the above, the two vertices in the same part are remote. If $|E| > \lfloor \frac{1}{4}(n-1)^2 \rfloor$ we get a contradiction to the statement of case 1 of the lemma. This proves that $\ex(n,\AAA\BB\CC) = \lfloor \frac{1}{4}(n-1)^2 \rfloor$.

It remains to show that if $|\FF| = \lfloor \frac{1}{4}(n-1)^2 \rfloor$ then $\FF$ is of the form $\FF^*_n$. For this, it suffices to show that all triples in $\FF$ share some fixed vertex $x$. Indeed, by the absence of $\CC$, the link of such $x$ in $\FF$ has to be a triangle-free graph on $V \setminus \{x\}$ with $\lfloor \frac{1}{4}(n-1)^2 \rfloor$ edges, hence it must be the Tur\'an graph.

Let $\overline{G}$ be the bipartite complement of $G$, i.e., $E(\overline{G})$ is the set of pairs in $A \times B$ which are not edges of $G$. As $|A|+|B|=n$,
we have $|E(\overline{G})| = |A| \cdot |B| - \lfloor \frac{1}{4}(n-1)^2 \rfloor \le \min\{ |A|,|B| \}$. On the other hand, if there is a triple in $\FF$ with two vertices in $A$ then they are remote in $G$, hence every vertex in $B$ is incident to an edge of $\overline{G}$. Similarly, if there is a triple in $\FF$ with two vertices in $B$ then every vertex in $A$ is incident to an edge of $\overline{G}$.

Assume first that there are triples of both kinds. We conclude that $|A| = |B|$ and $\overline{G}$ is a perfect matching between $A$ and $B$. Now, if $|A| = |B| \ge 3$ then no two vertices of the same part are remote and we get a contradiction. The only possibility is $|A| = |B| = 2$, in which case $n=4$, $|\FF| = 2$ and the two triples in $\FF$ must share two vertices.

Thus, we may assume that all triples in $\FF$ are of the same kind, having w.l.o.g.\ two vertices in $A$ and one in $B$. Let $x \in A$ and $y \in B$ form an edge of $\overline{G}$. Since every vertex in $B$ is incident to an edge of $\overline{G}$, and there are at most $|B|$ such edges, $xy$ is the only edge of $\overline{G}$ incident to $y$. Hence, no two vertices in $A \setminus \{x\}$ are remote. Therefore, every triple in $\FF$ must contain $x$, completing the proof.
\qed

\section{Excluding $\AAA$ and $\CC$}\label{sec:AC}
\message{Excluding $Aa$ and $CC$}

\subsection{$\AAA\CC$-free constructions}\label{subsec:AC}

We have already encountered two different equally large constructions of $\AAA\CC$-free triple systems: $\FF_T$, the one based on the transitive tournament (which also avoids $\DD$), and $\FF^*_n$, the one based on the Tur\'an graph (which also avoids $\BB$), both of size $\lfloor \frac{1}{4}(n-1)^2 \rfloor$. It turns out that these are optimal $\AAA\CC$-free constructions for even $n$, but we can do better by $1$ for odd $n \ge 5$.

Given an $\AAA\CC$-free family $\FF$ on the $n$-element vertex set $V$ one can define another $\AAA\CC$-free family $\FF^{+2}$ on the vertex set $V\cup \{ 1,2\}$ as follows. $\FF^{+2}:=\FF\cup \{ 12v: v\in V \}$. We call this new family a {\em $2$-extension} of $\FF$. One can define an $\AAA\CC$-free family $\FF^{+4}$ on the vertex set $V\cup \{ 1,2,3,4\}$ as
$$
  \FF^{+4}:=\FF\cup \{ 12v: v\in V \}\cup \{ 34v: v\in V \}\cup \{ 123, 234 \}.$$
Note that if $V \ne \emptyset$ then $(\FF^{+2})^{+2}$ is not the same as $\FF^{+4}$, although they only differ by replacing $124$ by $234$.
For even $n\geq 2$ let $\FFF_{(n)}$ denote the class of $n$-vertex triple systems one can obtain by an arbitrary sequence of $2$-extensions and $4$-extensions starting with a system of the form $\FF^*_{n-2k}$ for some $0 \le k \le \frac{n-2}{2}$ (the case $k=0$ gives $\FF^*_n$ without extensions, the case $k=\frac{n-2}{2}$ applies a sequence of extensions starting from the empty triple system on $2$ vertices and gives $\FF_T$ when it uses only $2$-extensions). All families $\FF \in \FFF_{(n)}$, $n$ even, have the same size $|\FF| = \lfloor \frac{1}{4}(n-1)^2 \rfloor$.

For $n=5$ we arrange the vertices $1,2,3,4,5$ cyclically, and let $\FF^5_5$ consist of the five triples of consecutive vertices $123, 234, 345, 451, 512$. We call such a triple system a \emph{five-ring}, and note that it contains no $\AAA$ or $\CC$. Its size is $5 = \frac{1}{4}(5-1)^2 + 1$.
For odd $n\geq 5$ the members of the class $\FFF_{(n)}$ are those $n$-vertex triple systems obtained from a five-ring by an arbitrary sequence of $2$- and $4$-extensions.
By induction we have $|\FF| = \frac{1}{4}(n-1)^2 + 1$ for all   odd $n \ge 5$ when $\FF\in \FFF_{(n)}$.

\begin{thm} \label{thm:AC}
Define
\begin{equation*}
f(n) :=
    \begin{cases}
        \lfloor\frac{1}{4}(n-1)^2\rfloor & \text{if } n \text{ is even or } n=1,3\\
        \frac{1}{4}(n-1)^2+1 & \text{if } n \text{ is odd, } n\neq1,3
    \end{cases}
\end{equation*}
Then $\ex(n,\AAA\CC)=f(n)$.
Moreover, for $n\geq 4$ the only extremal $\AAA\CC$-free families are the members of $\FFF_{(n)}$.
\end{thm}

It is easy to calculate that for $n$ even $\FFF_{(n)}$ has exactly $2f_{n/2}-1$ members (up to isomorphism), where $f_1,f_2, \dots$ is the Fibonacci sequence defined as $f_1=f_2=1$ and $f_k=f_{k-1}+f_{k-2}$ for $k>2$.
For odd $n\geq 5$ the size of $\FFF_{(n)}$ can also be obtained from a Fibonacci number, $|\FFF_{(n)}|=f_{(n-3)/2}$.
The multitude of extremal constructions may explain why the proof of Theorem~\ref{thm:AC} turned out to be more complicated than other proofs in this paper. Before turning to the proof itself, we further investigate the structure of possible maximal families and we also need a graph theoretic lemma (presented in Subsection~\ref{ss95}).

\subsection{Reducing $\AAA\CC$-free families}\label{subsec:ACreduction}

Suppose that the $\AAA\CC$-free triple system $\FF$ on $n$ vertices has a pair $u,v\in V$ such that $\deg_\FF(uv)=n-2$ and let $\HH$ be the family of triples of $\FF$ contained in $V\setminus \{ u,v\}$. Then $\FF= \HH^{+2}$.
This observation leads to the following useful lemma.

\begin{lem} \label{lem:reduction}
Let $\FF$ be an $\AAA\CC$-free triple system on the $n$-vertex set $V$.
Suppose that there are disjoint pairs $e_1, \dots, e_k\subset V$ ($1\leq k< \frac{n}{2}$), $K:= \cup e_i$,
 having the following properties:
\begin{enumerate} 
\item[{\rm (K1)}] There is no triple $F\in \FF$ such that $|F\cap K|=1$.
\item[{\rm (K2)}] If $F\in \FF$ meets both $e_i$ and $e_j$ then $F\subset e_i\cup e_j$
 (for all $1\leq i\neq j\leq k$).
 \end{enumerate}
Let $\HH$ denote the family of triples of $\FF$ contained in $V\setminus K$. Then
\begin{equation}\label{eq:redu}
  |\FF|- \lfloor\frac{1}{4}(n-1)^2\rfloor \leq  |\HH|- \lfloor\frac{1}{4}(n-2k-1)^2\rfloor,
  \end{equation}
and here equality holds if and only if
there is a sequence of $2$- and $4$-extensions of $\HH$ resulting in $\FF$.
\end{lem}

If $\FF$ possesses such a set of pairs, then we call it {\em reducible}.

\medskip
\noindent{\bf Proof of Lemma~\ref{lem:reduction}.}\quad
The family $\FF\setminus \HH$ contains two kinds of triples,
those meeting $V\setminus K$ and containing an $e_i$ and those contained in some $e_i\cup e_j$.
Since every four vertices can carry at most two triples we get
$$   |\FF|-|\HH| \leq k(n-2k) + 2 \binom{k}{2} = \lfloor\frac{1}{4}(n-1)^2\rfloor - \lfloor\frac{1}{4}(n-2k-1)^2\rfloor.
  $$
We are done with inequality~\eqref{eq:redu}.

In the case of equality we have $e_i \cup \{v\} \in \FF$ for all $1 \le i \le k$ and $v \in V \setminus K$, and
each $e_i\cup e_j$ contains exactly two triples from $\FF$.
Define an oriented graph $T_K$ with vertex set $\{ e_1, \dots, e_k\}$
 as follows. Add an arrow from $e_i$ to $e_j$ if there is a triple $F$ containing $e_i$ and meeting $e_j$.
Note that $T_K$ has at least one edge between $e_i$ and $e_j$, and it also might have two edges in opposite directions.
However, $T_K$ does not contain oriented triangles because $\FF$ is $\AAA$-free.
This implies that the pairs $\{ e_i, e_j\}$ with double arrows are pairwise disjoint, so they form a matching, say $L$.
We obtain that $T_K$ consists of a transitive tournament (say, there is an arrow from $e_i$ to $e_j$ for $1\leq j < i \leq k$)
together with a few reversed arrows $e_{i}\to e_{i+1}$ for $\{ e_{i},e_{i+1}\} \in L$.
Every $e_i$ outside $L$ corresponds to a 2-extension, and the pairs from $L$ define 4-extensions. \qed

\subsection{The structure of $\AAA\CC$-free families: partitioning the triples and the pairs}  


Let $\FF$ be an $\AAA\CC$-free triple system on $[n]$.
Recall that a five-ring is a triple system consisting of $5$ triples on $5$ vertices that is isomorphic to $\FF^5_5$ (defined in Subsection~\ref{subsec:AC}). Five-rings that are contained in the triple system $\FF$ will play a special role in our analysis. Suppose that $\PP \subseteq \FF$ is a five-ring. Then $V(\PP)$ contains no more triples from $ \FF \setminus \PP$.
Also, the case of $|F\cap V(\PP)|=2$ for some $F \in \FF$ would create an $\AAA$ configuration.
We conclude that triples in $\FF \setminus \PP$ can meet $V(\PP)$ in at most one vertex.
So the 10 pairs of $\PP$ are contained only in the triples from $\PP$.
Let $\RR\subseteq \FF$ be the family of triples contained in five-rings and let $G_R$ be the graph on $[n]$ whose edges are the pairs contained in members of $\RR$. If needed we split $G_R$ into two graphs, $G_{R1}$ and $G_{R2}$ to indicate if an edge is covered once or twice. We have $|\RR|=|E(G_{R1})|=|E(G_{R2})|$ and $|E(G_R)|= 2|\RR|$.


Every triple $F\in \FF$ contains at least one own pair.
We partition $\FF\setminus \RR$ into $\FF_1\cup\FF_2$ where $\FF_1$ consists of the triples with exactly one own pair, and $\FF_2$ of those with at least two own pairs.
Let $G_1$ be the graph on $[n]$ whose edges are the own pairs of the triples in $\FF_1$.
Picking exactly two own pairs of every triple in $\FF_2$, let $G_2$ be the graph on $[n]$ formed by them. Thus $|E(G_1)|=|\FF_1|$ and $|E(G_2)|=2|\FF_2|$.
As all edges in $G_1$ and $G_2$ are own pairs of some $F\in {\mathcal F}$
 we have $E(G_1)\cap E(G_2)=\emptyset.$


Define the graph ${G}_3$ as
\begin{equation*}
 E({G}_3) = \{xy\colon x,y \text{ are joined by a path of length $2$ in }  G_1\}. 
\end{equation*}
We will require the following fact about the edges of $G_3$.  

\begin{lem} \label{lem:G3}
If $xy \in E(G_3
 )$ then no triple in $\FF$ contains both $x$ and $y$.
\end{lem}

\begin{proof}  
Let $v$ be a vertex such that $xv,vy \in E(
{G}_1)$, and let $F_1$ and $F_2$ be the triples in $
{\FF}_1$ having $xv$ and $vy$, respectively, as their unique own pair. Assume for the sake of contradiction that the triple $F_3$ in $\FF$ contains $x$ and $y$. Clearly, $xvy$ cannot be a triple in $\FF$ because it contains both $xv$ and $vy$. Therefore $F_1,F_2,F_3$ are all distinct and each of them contains a different pair from $\{x,v,y\}$ and a vertex outside it. Say $F_1 = xvw$, $F_2 = vyz$, $F_3 = xyu$. To avoid forming an $\AAA$ or a $\CC$, two of the vertices $w,z,u$ must be the same and the third must be different.

We consider first the case when $w=z$. Since $xw$ is not an own pair of $F_1$, there is another triple in $\FF$ containing it, say $F_4 = xws$. If $s \notin \{y,u\}$ then $F_2,F_3,F_4$ form a forbidden $\AAA$. If $s=y$ then $F_1,F_2,F_4$ form a forbidden $\CC$. Thus $s=u$ is the only possibility. Repeating the same arguments for $yw$ which is not an own pair of $F_2$, we find the triple $F_5 = ywu$ in $\FF$. But now $F_3,F_4,F_5$ form a $\CC$, giving a contradiction.

Thus, we are left with the case when $w \ne z$ and $u$ is equal to one of them, w.l.o.g.\ $u=w$. Since $vw$ is not an own pair of $F_1$, there is another triple in $\FF$ containing it, say $F_4 = vws$. If $s \notin \{y,z\}$ then $F_2,F_3,F_4$ form a forbidden $\AAA$. If $s=y$ then $F_1,F_3,F_4$ form a forbidden $\CC$. Thus $s=z$ is the only possibility. Since $yz$ is not an own pair of $F_2$, there is another triple in $\FF$ containing it, say $F_5 = yzt$. If $t \notin \{x,w\}$ then $F_3,F_4,F_5$ form a forbidden $\AAA$. If $t=w$ then $F_2,F_4,F_5$ form a forbidden $\CC$. Thus $t=x$ is the only possibility. Summarizing, we have the following five triples in $\FF$: \[ F_1 = vwx, F_3 = wxy, F_5 = xyz, F_2 = yzv, F_4 = zvw. \]
They form a five-ring, contradicting the fact that $F_1$ and $F_2$ are in $
{\FF}_1$, hence not part of a five-ring.
\end{proof}  

 \vspace{10pt}

Our Lemma~\ref{lem:G3} implies
 that $E(G_3)$ is disjoint from $E(G_R)\cup E(G_1)\cup E(G_2)$, so the following four 
  edge sets are pairwise disjoint:
$$
E(
{G}_R), \;E(
{G}_1), \;E(G_2), 
\;E(
G_3).
$$
By adding a fifth 
one
$$
E(G_4) := \{xy\colon x\neq y, xy \text{ is in none of the above}\},
$$
we obtain a partition of
$E(K_n)$ into 
five graphs.
Figure~\ref{fig:tripletypes} shows, for every possible triple in $\FF$, to which of the above graphs its pairs belong.
Again, if needed, we split $G_4$ into three graphs, $G_{40}$, $G_{41}$ and $G_{42}$ to indicate if an edge is not covered, covered once or at least twice by members of $\FF$, respectively.

\begin{figure}
\centering
\includegraphics[width=0.4\textwidth]{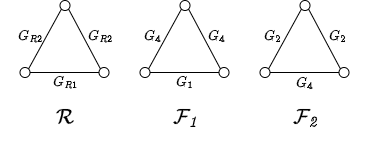}
\caption{\label{fig:tripletypes}Assignment of the pairs of each triple in $\FF$ to graphs}
\end{figure}

The discussion above yields the following observations, to be used later:
\begin{align}
\label{eq:geq2}
    \deg_{\FF}(xy)>2 & \:\Rightarrow\: xy\in E(G_{42})\\
\label{eq:eq2}
    \deg_{\FF}(xy)=2 & \:\Rightarrow\: xy\in E(
    G_{R2})\cup E(G_{42})\\
\label{eq:eq0}
    \deg_{\FF}(xy)=0 & \:\Leftrightarrow\: xy\in E(
    G_3)\cup E(G_{40})
\end{align}

\subsection{{Proof of Theorem~\ref{thm:AC}}: three cases without induction}

Let $\FF$ be an $\AAA\CC$-free triple system on $[n]$ and suppose that  $|\FF|=\ex(n,\AAA\CC)$.
Constructions showing $|\FF|\geq f(n)$ were given above.
It is trivial to check their optimality for $n\leq 4$, so from now on we may suppose that $n\geq 5$.
We will prove that $|\FF|\leq f(n)$ and $\FF\in \FFF_{(n)}$.
We break the argument into six cases.

\subsubsection{Case 1: There is an isolated vertex in $G_4$}\label{sss:941}
Here we assume that $x$ is an isolated vertex in $G_4$.
In this case all edges $xy$ belong to either $G_2$ or $G_R$. Indeed, $xy\notin E(G_4)$. Also $xy\in E(G_1)$ implies that there is an $xzy\in \FF_1$, so there is a path $x$-$z$-$y$ in $G_4$. But this is impossible since there is no $xz\in E(G_4)$, so $xy\notin E(G_1)$.
If there is no $xy$ edge in $E(G_1)$  then there is no $xy\in E(G_3)$ either.

Consider the graph $\FF[x]$, the link of $x$ in $\FF$.
If the pair $xy$ is contained in $d$ triples (i.e., $\deg_\FF(xy)=d$) then the degree of the vertex $y$ in the graph $\FF[x]$ is exactly $d$.
Since $xy\not\in E(G_4)$ all the degrees in $\FF[x]$ are at most 2 by observation~(\ref{eq:geq2}).
Moreover, if a component $C$ of $\FF[x]$ has a vertex of degree 2 then by observation~(\ref{eq:eq2}), $V(C)\cup\{x\}$ is the vertex set of a five-ring, and $C$ is a path of length $3$. Thus, any component $C$ of $\FF[x]$ that is not an isolated vertex is of one of two forms:


\begin{enumerate}[label=\Roman*]
    \item A single edge $e$. \\We let $k\geq0$ be the number of such components, and denote them by $e_1, ..., e_k$. We denote by $K$ the set of vertices of these components, $|K|=2k$.
    \item A path $P$ of length $3$. \\We let $\ell\geq0$ be the number of such components, and denote them by $P_1, ..., P_\ell$. We denote by $L$ the set of vertices of these components, $|L|=4\ell$.
\end{enumerate}
Note that $xy\in E(G_2)$ for any $y\in K$, and $xy\in E( 
 G_R)$ for any $y\in L$.
We claim that there are no isolated vertices in $\FF[x]$. Indeed, suppose that $z$ is an isolated vertex. Then $\deg_{\FF}(xz)=0$ and hence by observation~(\ref{eq:eq0}), $xz\in E( 
G_3)\cup E(G_{40})$. This contradicts the observation above that all edges incident to $x$ belong to either $G_2$ or $G_R$.

Thus, $\FF[x]$ gives the following decomposition of $[n] \setminus \{x\}$: \[ [n] \setminus \{x\} = K \cup L, \text{  where  } |K|=2k,\; |L|=4\ell. \]
Note that $n$ must be odd in this case since $n=|K|+|L|+1$. 
Next, we bound the number of triples contained in $K\cup L$:
\begin{itemize}
    \item[-] Each $P_i$ contains two triples (those triples in the five-ring that do not include $x$). Other than those, no triple in $K\cup L$ contains two vertices of the same $P_i$.
    \item[-] To avoid an $\AAA$ configuration, no triple in $K\cup L$ meets three components of $\FF[x]$,
       and no triple meets $P_i$ and $P_j$ when $i\neq j$.
    \item[-] Thus, any triple in $K\cup L$ that is not contained in a $P_i$, consists of both vertices of some $e_i$, and one vertex of another component.
    \item[-] To avoid a $\CC$ configuration, for every $i\neq j$ there are at most $2$ triples contained in $e_i \cup e_j$.
\end{itemize}

In total, we have
\begin{equation}\label{eq:96}
    |\FF|\leq k+5\ell +2 {{k}\choose{2}} +k\cdot 4\ell=(k+2\ell)^2+5\ell-4\ell^2\leq \left(\frac{n-1}{2}\right)^2+1=f(n),
  \end{equation}
which completes the proof of the upper bound in Case 1.

If equality holds in~\eqref{eq:96} then $\ell=1$, $k=\frac{n-5}{2}$. If $k=0$ then $n=5$ and $\FF$ is a five-ring, i.e., $\FF \in \FFF_{(5)}$. Suppose now that $k \ge 1$. According to the above observations $K$
and $\FF$ satisfy the two constraints of Lemma~\ref{lem:reduction}, so $\FF$ is reducible.
Inequality~\eqref{eq:redu} and $|\FF|=f(n)$ yield
$$
 |\FF|- \lfloor\frac{1}{4}(n-1)^2\rfloor =  |\HH|- \lfloor\frac{1}{4}(n-2k-1)^2\rfloor 
 =|\HH|- 4 =1
$$
where $\HH$ is a five-ring.
Then Lemma~\ref{lem:reduction} gives that $\FF$ is obtained from $\HH$ by a sequence of $2$- and $4$-extensions, i.e., $\FF\in \FFF_{(n)}$. \qed

\subsubsection{Case 2: There is a pendant $G_{4}$ edge which is in $G_{40}$}\label{sss:942}

Let $x$ be a degree one vertex of $G_4$, and let $y$ be its unique neighbor in $G_4$.
In this case we suppose that $\deg_\FF(xy)=0$, the pair $xy$ does not occur in any triple of $\FF$.
Again, as in Case 1, there is no $G_1$ edge incident to $x$, so there is no such $G_3$ edge either.

Consider the structure of the link $\FF[x]$, similar to what we did in Case 1.
The vertex $y$ is isolated in $\FF[x]$ and  $[n] \setminus \{x\} = K\cup L \cup \{y\}$.
In particular $n$ is even here, $n=2k+4\ell+2$. The only possible triples in $\FF$ that contain $y$ and do not give rise to an $\AAA$ configuration are of the form $e_i\cup\{y\}$. There are at most $k$ of those.
Thus, like in~\eqref{eq:96}, we get
\begin{equation*}\label{eq:97}
    |\FF|\leq k+5\ell +2 {{k}\choose{2}} +k\cdot 4\ell +k= \lfloor \left(\frac{n-1}{2}\right)^2\rfloor +3\ell-4\ell^2\leq f(n).
  \end{equation*}
If equality holds here then $\ell=0$, $k=\frac{n-2}{2}$, and Lemma~\ref{lem:reduction} can be applied to obtain $\FF\in \FFF_{(n)}$. \qed

\subsubsection{Case 3: There is a pendant $G_{4}$ edge which is in $G_{41}$}\label{sss:943}

Let $x$ be a degree one vertex of $G_4$, and let $y$ be its unique neighbor in $G_4$.
In this case we suppose that $\deg_\FF(xy)=1$, the pair $xy$ occurs in a triple $xyz\in \FF$ with
  $\deg_\FF(xz)=\deg_\FF(yz)=1$.
In the process creating the graph $G_2$ we selected $xz$ and $yz$ and put $xy$ to $G_{41}$.
Redefine these graphs on these three pairs, put $yz$ into $G_{41}$ and the other pairs to $G_2$. Then $x$ becomes isolated in the new $G_4$. That case has been solved as Case 1.  \qed



\subsection{
Pairs of vertices joined by a path of length $2$}\label{ss95}
Given a graph $G$, we define a new graph $T=T(G)$ on the same vertex set as follows: two vertices $x \ne y$ form an edge in $T$ if they are joined by a path of length $2$ in $G$, i.e., there exists a vertex $z$ such that both $xz$ and $zy$ are edges of $G$.

\begin{lem} \label{lem:G13}
Let $G$ be a graph on $n$ vertices, and let $T=T(G)$ be
the graph defined above.
Then
\begin{equation} \label{eq:G13}
|E(T)|\geq |E(G)|-\lfloor \frac{n}{2} \rfloor.
\end{equation}
This inequality holds with equality if and only if one of the following is the case:
\begin{enumerate}[label=(\alph*)]
\item[{\rm (L1)}] $n$ is even, and $G$ is the disjoint union of balanced complete bipartite graphs $K(A_i, B_i)$ with $|A_i|=|B_i|=d_i$, $i=1,...,r$ where $d_i\geq1, r\geq1, \sum_{i=1}^r d_i=\frac{n}{2}$.
\item[{\rm (L2)}] $n$ is odd, and $G$ is as above except that one of the components is of the form $K(A_i, B_i)$ with $|A_i|=d_i$ and $|B_i|=d_i+1$ or is an isolated vertex.
\end{enumerate}
Moreover, if
$n$ is even and $|E(T)|= |E(G)|-\frac{n}{2} +1$, then either
\begin{enumerate}[label=(\alph*2)]
\item[{\rm (L3)}]  one component of $G$ is a $K_{d,d}$ with one edge left out, $d\geq 2$ (denote this as $K_{d,d}^-$),
  and the other components are balanced complete bipartite graphs, or
\item[{\rm (L4)}]
  there are positive integers  $d_1, d_2, \dots, d_r$, $r\geq2$ such that $\sum_{i=1}^r d_i=\frac{n}{2}+1$
   and the components of $G$ are complete bipartite graphs $K_{d_1, d_1-1}$, $K_{d_2, d_2-1}$, and $K_{d_i,d_i}$ for $i>2$.
\end{enumerate}
\end{lem}

The inequality~\eqref{eq:G13} was proved in dual form by F\"uredi~\cite{FZ:diam2}. Below we also prove the characterization of equality and at the same time re-prove the inequality.

Note that in each of the cases~(L1)--(L4) the graph $T(G)$ is a disjoint union of complete graphs. Also, in cases~(L1) and~(L3) the graph $G$ has no isolated vertex.

\vspace{10pt}

\noindent{\bf Proof of Lemma~\ref{lem:G13}.}\quad
The graphs described in (L1) and (L2) satisfy inequality~(\ref{eq:G13}) with equality.
We show 
 that all other graphs satisfy~{\eqref{eq:G13}} with strict inequality.

For a graph $H$ and a vertex $v$, we denote by $N_H(v)$ the set of neighbors of $v$ in $H$. We observe that for every vertex $x$,
\begin{equation*}\label{eq:lem1e1}
N_T(x) = \bigcup_{y\in N_G(x)} \left( N_G(y)\setminus\{x\}\right).
\end{equation*}
In particular, if $x$ has a neighbor $y$ with $\deg_G(x)<\deg_G(y)$ then $\deg_T(x)\geq \deg_G(x)$. This implies that there always exists a vertex $x$ with $\deg_T(x)\geq \deg_G(x)$, unless every connected component $C_i$ of $G$ is $d_i$-regular for some $d_i\geq 1$, and moreover, in $C_i$ the $d_i$ neighbors of $x$ all have the same $d_i-1$ neighbors in addition to $x$, which means that $C_i\cong K_{d_i, d_i}$. In other words, if $G$ is not of the form indicated in (L1), then there exists a vertex $x$ with $\deg_T(x)\geq \deg_G(x)$.

We now proceed 
  by induction on $n$. The base case $n=1$ is trivial. For the induction step, we assume that $G$ and $T=T(G)$ are graphs on $n \ge 2$ vertices, and the lemma is true for $n-1$.

Assuming that $G$ is not of the form indicated in (L1), 
 let $x$ be a vertex with $\deg_T(x)\geq \deg_G(x)$.
Applying the induction hypothesis to $G\setminus\{x\}$ we get 
\begin{equation}\label{eq:ind}
|E(T(G\setminus\{x\}))|\geq |E(G\setminus\{x\})|-\lfloor\frac{n-1}{2}\rfloor.
\end{equation}
Therefore
\begin{equation}\label{eq:sum}
\begin{split}
|E(T)| & \geq |E(T(G \setminus \{x\}))| + \deg_T(x)\\ & \geq |E(G \setminus \{x\})| - \lfloor \frac{n-1}{2} \rfloor + \deg_G(x) = |E(G)| - \lfloor \frac{n-1}{2} \rfloor.
\end{split}
\end{equation}

When $n$ is odd, equality in~(\ref{eq:G13}) is possible, but it requires that all the above inequalities hold with equality. In particular, we must have equality in~(\ref{eq:ind}), hence by induction $G \setminus \{x\}$ has to be of the form (L1).
Consider $N_G(x)$. If it meets two or more components of $G \setminus \{x\}$, or two sides of the same component, then $T(G)$ has at least one edge with vertices in $G \setminus \{x\}$ that is not in $E(T(G \setminus \{x\}))$. This makes the first inequality in~(\ref{eq:sum}) strict. We conclude that there is a component $K(A_i,B_i)$ of $G \setminus \{x\}$ with $|A_i| = |B_i| = d_i$ such that $N_G(x)$ is contained in one of its parts, say in $A_i$. If $\emptyset \subsetneq N_G(x) \subsetneq A_i$ then $\deg_T(x) = d_i > \deg_G(x)$, making the second inequality in~(\ref{eq:sum}) strict. This leaves only two possibilities: either $N_G(x) = A_i$, in which case we can add $x$ to $B_i$ and get a $K_{d_i,d_i+1}$, or $N_G(x) = \emptyset$ making $x$ an isolated vertex in $G$. In either case, $G$ fits the description in part (L2) of the lemma.

If $n$ is even then $\lfloor \frac{n-1}{2} \rfloor < \lfloor \frac{n}{2} \rfloor$, hence~(\ref{eq:sum}) implies that inequality~(\ref{eq:G13}) holds strictly. This completes the proof of~\eqref{eq:G13} together with (L1) and (L2).

Consider now the case $|E(T)|= |E(G)|-\frac{n}{2} +1$.
Again equality must hold in~\eqref{eq:ind} and~\eqref{eq:sum} so $\deg_T(x) =\deg_G(x)$, $G \setminus \{x\}$ has to be of the form (L2). If $N_G(x) = \emptyset$ then $G$ is of the form (L4). If $N_G(x)\neq \emptyset$ then it meets only one component $C$. This component could not be an isolated vertex, because that would give $\deg_G(x)=1$, $\deg_T(x)=0$. So $C=K(A_i,B_i)$ and $N_G(x)$ is a subset of one side of $C$ while $N_T(x)$ is the entire other side. Given that $\deg_T(x) = \deg_G(x)$, there are two possibilities. If $|A_i|=|B_i|$ then $N_G(x)$ is a full side of $C$, say $B_i$, so we get $K(A_i \cup \{x\}, B_i)$, yielding case (L4). If $|A_i|=|B_i| - 1$ then $N_G(x) = B_i \setminus \{y\}$ for some $y \in B_i$, so we get $K(A_i \cup \{x\}, B_i)$ minus $xy$, yielding case (L3).
\qed

\subsubsection{Remark}

We do not use but can prove an even stronger version of Lemma~\ref{lem:G13}.

\begin{claim} \label{claim:821}
Let $G$ be a graph, and let $T=T(G)$ be
the graph defined above.
Suppose that $\{ e_1, \dots, e_\ell\}$ is a maximal matching of $G$, i.e., $A:= V(G)\setminus (\cup e_i)$ is an independent set of vertices.
Then
\begin{equation} \label{eq:G132}
|E(T)|\geq |E(G)|-\ell.
\end{equation}
This inequality holds with equality if and only if each component of $G$ is either a balanced complete bipartite graph, or a $K_{d,d}^-$ (with $d\geq 2$) contributing $d-1$ edges to the maximal matching, or a $K_{d,d+1}$ ($d\geq 1$), or an isolated vertex.
\end{claim}

\noindent{\em Sketch of proof of Claim~\ref{claim:821}.}\quad
For two distinct edges $e_i, e_j$ in the maximal matching, let $E(G[e_i,e_j])$ denote the set of edges of $G$ with one end in $e_i$ and the other in $e_j$, and let $E(T[e_i,e_j])$ denote the analogous set for $T$. Similarly, for an edge $e_i$ and a vertex $a \in A$, let $E(G[e_i,a])$ and $E(T[e_i,a])$ denote the corresponding sets of edges with one end in $e_i$ and the other equal to $a$.

By inspection, we have $|E(T[e_i,e_j])| \geq |E(G[e_i,e_j])|$ and $|E(T[e_i,a])| \geq |E(G[e_i,a])|$ for any two edges $e_i,e_j$ and vertex $a\in A$. As this accounts for all edges of $G$ apart from $e_1,\ldots,e_{\ell}$, it implies~\eqref{eq:G132}.
In case of equality we must have equalities above and moreover, no $e_i$ and no two vertices in $A$ can form an edge in $T$. We obtain that
$E(G[e_i,e_j])$ should be either empty or two disjoint edges, and similarly $E(G[e_i,a])$ is either empty or a single edge.
The rest of the proof is just to check how these edges can fit together so as to maintain the above conditions for equality.
\qed

\subsection{The core inequality of the proof, 
    and more reductions in case of a pendant $G_4$ edge}

We will prove by induction on $n$ that $|\FF| \le f(n)$ and $\FF\in \FFF_{(n)}$.  

Look at the partition $[n]=C_1\cup \cdots \cup C_r$
where the $C_i$ are the connected components of $G_4$. By Case~1  we may suppose that $p_i:=|C_i|\geq 2$ for every $i$.
Note that every edge of $
G_1$ is contained in some $C_i$ 
 and the same is true for every edge of $
G_3$.
If $C$ is one of the $C_i$, we write $G_4^C, 
 G_1^C, 
 G_3^C$
for the corresponding graph restricted to $C$.
We have
\begin{equation}\label{eq:96core}
    \begin{split}
    2|\FF| & = 2|\RR| + 2|\FF_1| + 2|\FF_2|  \\
        & = |E(G_R)| + 2|E(G_1)| + |E(G_2)| \\
        & =|E(G_R)| + |E(G_1)| + |E(G_2)| + |E(G_3)| + |E(G_4)| -\frac{n}{2}  \\
     & {}\quad  +\left(
    n-|E( G_4)|\right) + \left( |E(G_1)| - |E(G_3)| -  \frac{n}{2}\right)  \\
     & = {n\choose2} -\frac{n}{2}
       + \sum_i \left( \left( p_i -|E( G^{C_i}_4)|\right) +  \left(|E(G_1^{C_i})| - |E(G_3^{C_i})| -  \frac{p_i}{2}\right) \right).
    \end{split}
\end{equation}
In the last line for every component $C_i$ we have
 $p_i- |E(G_4^{C_i})|\leq 1$ and  Lemma~\ref{lem:G13} applied to $G_3^{C_i} = T(G_1^{C_i})$ states  $ |E(G_1^{C_i})| - |E(G_3^{C_i})| - \lfloor \frac{p_i}{2}\rfloor \leq 0$.
So~\eqref{eq:96core} yields a pretty good upper bound for $|\FF|$ and to improve it further we will investigate the components where
  $G_4^{C_i}$ is a tree, and also the cases when $|E(G_1^{C_i})| - |E(G_3^{C_i})| -  \frac{p_i}{2}\geq -1$.

Call a component $C$ of size $p$ {\em \nice} if
\begin{equation}\label{eq:961}
  |E(G_3^C)|\leq |E(G_1^C)|-\frac{p-2}{2}.
  \end{equation}
For {\em not} \nice components the quantity
$$
 \sigma(C_i):=  \left( p_i -|E( G^{C_i}_4)|\right) +  \left( |E(G_1^{C_i})| - |E(G_3^{C_i})| -  \frac{p_i}{2} \right)
 $$
is at most $-\frac{1}{2}$.
Since $|\FF|\geq f(n)$ we must have at least one \nice component.
In such a component, Lemma~\ref{lem:G13} and~\eqref{eq:961} imply that $G_3^C$ consists of vertex disjoint complete graphs (including isolated vertices, of course).

\subsubsection{Case 4:  There is a pendant $G_4$ edge in a \nice component}\label{sss:961}

In this subsection we suppose that $C$ is a \nice component and
 $x \in C$ is a degree one vertex of $G_4$. Let $y$ be its unique neighbor in $G_4$.
Since we have already discussed the case $\deg_\FF(xy)=0$ (Case~2) and $\deg_\FF(xy)=1$ (Case~3)
 we may suppose that $\deg_\FF(xy)=:q\geq 2$.

Consider the structure of the link $\FF[x]$, similar to what we did in Case 1.
There are two further forms of components in $\FF[x]$ in addition to forms I and II which appear in $K$ and $L$, respectively.

\begin{enumerate}[label=\Roman*]
    \setcounter{enumi}{2}
    \item A star $S_q$ centered at $y$ and having $q$ leaves, $q\geq 2$. \\There is exactly one such component in $\FF[x]$. The leaves of $S_q$ can be partitioned into two (possibly empty) parts:
    \begin{equation*}
    \begin{split}
    U & \coloneqq\{u\colon xyu \in 
\FF_1, xu\in E(
G_1), yu\in E(G_4)\},\\
        V & \coloneqq\{v\colon xyv \in \FF_2, xv, yv\in E(G_2)\}.
    \end{split}
    \end{equation*}
    We write $|U|=s$, $|V|=t$, $s+t=q$.
    \item An isolated vertex $z$. \\We let $m\geq0$ be the number of isolated vertices, and denote by $M$ the set of isolated vertices. We remark that isolated vertices in $\FF[x]$ are now permissible because $x$ has neighbors in $
G_1$ (provided that $s>0$), which allows the isolated vertices to be neighbors of $x$ in $
G_3$.
\end{enumerate}
We summarize: $\FF[x]$ looks as in Figure \ref{fig:HxStructure}, and $[n] \setminus \{x\} = K \cup L \cup M \cup \{y\} \cup U \cup V$ with
$K=e_1\cup \dots \cup e_k$ and $L=P_1\cup \dots \cup P_\ell$.

\begin{figure}
\centering
\includegraphics[width=0.7\textwidth]{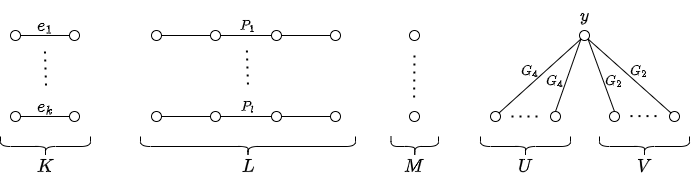}
\caption{\label{fig:HxStructure} The structure of $\FF[x]$ in 
 Case 4}
\end{figure}

\bigskip
Our next arguments exploit and combine structural information about $\FF[x]$ and about $
G_1^C, 
G_3^C,G_4^C$ based on Lemma \ref{lem:G13}.
Since the $G_1$ and $G_3$ neighbors of $x$ are in the same component, we already know that

\begin{itemize}
\item[(C1)] $\{x, y\}, M, U\subseteq C$.
\end{itemize}
Since $G_3^C$ consists of cliques and the $G_3$ neighborhood of $x$ is exactly $M$ we get
\begin{itemize} \item[(C2)] $M\cup \{x\}$ is a (maximal) clique in $G_3$.
\end{itemize}
Similarly, since $U$ is the $G_1$ neighborhood of $x$, by definition
\begin{itemize}
 \item[(C3)] $U$ is a clique in $G_3$.
\end{itemize}
 It follows from Lemma~\ref{lem:G3} that no two vertices in $M\cup \{x\}$ belong together to a triple in $\FF$ and similarly
  $U$ meets every triple in $\FF$ in at most one vertex.
\begin{itemize}
\item[(C4)] We may assume that $K=\emptyset$.
\end{itemize}

\noindent {\em Proof.}\enskip
Suppose that $k \ge 1$. First we show that the constraints (K1) and (K2) of Lemma~\ref{lem:reduction} are valid for $\FF$.
Suppose, on the contrary to (K1), that there is a $vwz\in \FF$, $K\cap vwz=\{ v\}$, $v\in e_i$.
We have $x\notin vwz$, because we know $\FF[x]$.
If $wz\cap L\neq \emptyset$ then we get an $\AAA$ configuration.
If one of $w,z$ is $y$, say $w=y$, we can choose a $z' \in U \cup V$, $z' \ne z$ (since $|U \cup V| \ge 2$) and find the $\AAA$ configuration $e_i \cup \{x\}, vyz, xyz'$. Thus we may assume that $w,z \in M \cup U \cup V$, and by (C2) at least one of them, say $z$, is in $U \cup V$. But now $e_i \cup \{x\}, vwz, xyz$ form an $\AAA$.
To check (K2) is obvious.
Applying~\eqref{eq:redu}, we get
\begin{equation*}\label{eq:redu2}
  \eps(n):=f(n) - \lfloor\frac{1}{4}(n-1)^2\rfloor \le |\FF|- \lfloor\frac{1}{4}(n-1)^2\rfloor \leq  |\HH|- \lfloor\frac{1}{4}(n-2k-1)^2\rfloor,
  \end{equation*}
implying $|\HH|\geq f(n-2k)-\eps(n-2k)+\eps(n)$. Thus $|\HH|\geq f(n-2k)$.
By induction, this must hold with equality and $\HH\in \FFF_{(n-2k)}$. Also~\eqref{eq:redu} holds with equality, so $\FF$ is obtained from $\HH$ by a sequence of $2$- and $4$-extensions, i.e., $\FF \in \FFF_{(n)}$.
\qed

\begin{itemize}
\item[(C5)] $L=\emptyset$.
\end{itemize}
\noindent {\em Proof.}\enskip
Suppose that $\ell \ge 1$. By a case check similar to the proof of (C4), one can verify that any triple $vwz \in \FF$ that meets $P_1$ in a single vertex creates an $\AAA$ configuration. Thus, the only triples in $\FF$ that meet $P_1$ are those in its five-ring. Let $\HH$ be the rest of the triples. We obtain
\begin{equation*}\label{eq:redu5}
  f(n) \le |\FF| = |\HH|+5 \leq f(n-4)+5.
  \end{equation*}
This is a contradiction, noting that under our conditions $n\geq 8$. \qed

\begin{itemize}
\item[(C6)] We may assume that $M\neq \emptyset$, therefore $U\neq \emptyset$.
\end{itemize}
\noindent Indeed, $M=\emptyset$ implies $[n]=\{x, y\}\cup U\cup V$ with $\deg_\FF(xy)=n-2$. As we have seen before Lemma~\ref{lem:reduction}, this implies that
  $\FF$ is a $2$-extension of $\FF\setminus \{ x,y\}$, and we are done by induction.

\begin{itemize}
\item[(C7)] $y$ is an isolated vertex in $
G_1^C$.
\end{itemize}
\noindent
Indeed, suppose that $yz\in E( 
G_1^C)$.  Clearly $z\in M$.
Every $z\in M$ has a $G_1$ neighbor in $U$, say $zw\in E(G_1)$.
Then, by definition, the two $G_1$ edges $yz$ and $zw$ place $yw$ in $G_3$. However, $yw$ is a $G_4$ pair.

    Since $y$ is isolated in $
    G_1^C$, according to the structure of $G_1^C$ described in  Lemma~\ref{lem:G13}, only cases (L2) or (L4) are possible, so every other component of $
    G_1^C$ is a (balanced or almost balanced) complete bipartite graph.
Consider the $G_1$ component $\widetilde C$ to which $x$ belongs. We know from $\FF[x]$ that $N_{
G_1^C}(x)=U$ and $N_{
G_3^C}(x)=M$. Thus, $\widetilde C$  
is a complete bipartite graph of $G_1$-edges with parts
$\{x\}\cup M$ and $U$.

Let us conclude this case with the following statement.
\begin{claim}\label{cl:96}
Suppose that $|C|=p$ is even and $C$ is a \nice component with a pendant $G_4$ edge. Then  equality holds in~\eqref{eq:961}, i.e.,
 $|E(G_1^{C})| - |E(G_3^{C})| -  \frac{p}{2}=-1$. \end{claim}
\noindent {\em Proof.}\enskip
By (C7) the graph $G_1^C$ has an isolated vertex (namely $y$). Then
(L1) is not possible in  Lemma~\ref{lem:G13}.
\qed

\subsubsection{Case 5:  The $G_4$ edges form a tree in a \nice component}\label{sss:962}

It is, in fact a subcase and a continuation of Case 4. We maintain the assumptions and conclusions of Case~4, and add the assumption that $G_4^C$ is a tree.
Again let $x$ be a degree one vertex of $G_4^C$, and let $y$ be its unique neighbor in $G_4$.

\begin{itemize}
\item[(C8)] The $G_4$ neighborhood of $y$ is $\{ x\} \cup M\cup U$.
\end{itemize}
\noindent
{\em Proof.}\enskip
We already know that the vertices in $\{x\} \cup U$ are in the $G_4$ neighborhood of $y$, and those in $V$ are not. As $K=L= \emptyset$, it remains to show that $M \subseteq N_{G_4}(y)$.

The arguments leading to (C7) can be carried out for any choice of a leaf 
of $G_4^C$. We may assume that all such leaves are incident to $G_{42}$ edges,  
otherwise we are done. Thus, for every leaf of $G_4^C$ its unique neighbor in $G_4^C$ is an isolated  $
G_1^C$ vertex.
If there is only one isolated vertex in $
G_1^C$ (namely $y$), then it must be that vertex, regardless of the choice of the leaf of $G_4^C$. This implies that $G_4^C$ is a star, and it is centered at $y$. In particular, $M \subseteq N_{G_4}(y)$.

Consider the case when $G_1^C$ has at least two isolated vertices, $y$ and $y'$.
The \niceness of $C$ and Lemma~\ref{lem:G13} imply that only the case (L4) is possible, there are no more isolated vertices.
By (C7) each leaf of the tree $G_4^C$ is adjacent to an isolated vertex of $G_1^C$.
So the tree $G_4^C$ consists of two stars with centers $y$ and $y'$ and a path connecting these centers.
In particular, all but one of the $G_4^C$ neighbors of $y$ are leaves. According to (L4) the non-isolated components of $G_1^C$ are balanced complete bipartite graphs, in particular such is $\widetilde C = K(\{x\} \cup M, U)$, so $|U| = |\{x\} \cup M| \ge 2$. All vertices in $U$ are $G_4^C$ neighbors of $y$, so at least one of them, say $u$, is a leaf in $G_4^C$ adjacent to $y$. Now consider any vertex $z \in M$. As $zu \in E(G_1^C)$, $z$ and $u$ must have a common $G_4^C$ neighbor, which can only be $y$. This shows that $z \in N_{G_4}(y)$, completing the proof.
\qed

\begin{itemize}
\item[(C9)] For every pair $ab$ with $a\in \{x\} \cup M$, $b\in U$, the unique triple in $\FF$ containing it is $aby$.
\end{itemize}
\noindent
 Indeed, suppose $abz$ is that unique triple.
Then $az, bz\in E(G_4^C)$, so $z$ is the middle vertex of a path of length $2$ in $G_4^C$.
Since $G_4^C$ is a tree, this path $a$-$z$-$b$ is the unique path connecting $a$ to $b$ in the tree.
But $ay$ and $by$ are tree edges by (C8) so $z= y$.
\qed

\begin{itemize}
\item[(C10)] $|U|\geq 2$.
\end{itemize}
\noindent
Indeed, by  Lemma~\ref{lem:G13} and the \niceness of $C$ we know that $K(\{ x\}\cup M,  U)$ is balanced or almost balanced. We have $(|M|+1) - |U| \leq 1$, so $|U|\leq 1$ gives $|M|\leq 1$.
Writing $M=\{z\}$, $U=\{u\}$, we have $zuy \in \FF$ by (C9). As $zu$ is the only own pair of this triple, there must be another triple $zyv \in \FF$. The only remaining options are $v=x$ or $v \in V$. Both are impossible: the former because $\deg_\FF(xz)=0$, the latter because the pair $yv$, $v \in V$, is an own pair of $xyv$.
\qed

\bigskip

Next, we give an upper bound on the number of triples in $\FF$. 



\begin{itemize}
    \item[-] Triples containing no vertex of $V$:\\
    Since $\{x\} \cup M$ and $U$ are cliques in $G_3$, any triple in $\FF$ can contain at most one vertex from each of them (Lemma~\ref{lem:G3}). Thus triples avoiding $V$ must be of the form $aby$ with $a \in \{x\} \cup M$, $b \in U$. By (C9) all these are indeed triples in $\FF$, so we have $(m+1)s$ of them.
    \item[-] Triples containing exactly one vertex of $V$:
    \\There are $t$ such triples containing $xy$. There are no other such triples. Indeed, any other such triple would have to be of the form $zuv$, $z\in M, u\in U, v\in V$. But then, taking $w\in U\setminus\{u\}$, we get an $\AAA$ configuration: $zuv,zwy,xyv$.
    \item[-] Triples containing exactly two vertices of $V$: \\There are no such triples. Indeed, suppose $zvw$ is such a triple, with $v, w \in V$. Then $z\neq x$, because $vw$ is not an edge of $\FF[x]$. If $z=y$ we get a $\CC$ configuration: $yvw, xyv, xyw$. If $z\in M$ then taking $u\in U$ we get an $\AAA$ configuration: $zvw, zuy, xyv$. Finally, suppose $z\in U$. Since $zy$ is not an own pair of $xyz$, there is another triple $zya\in\FF$ containing it. The vertex $a$ cannot be in $V$ (by the previous paragraph), so we get an $\AAA$ configuration: $zvw, zya, xyv$.
    \item[-] Triples contained in $V$: \\By the induction hypothesis, there are at most $f(t)$ such triples.
\end{itemize}

Altogether, we get that the number of triples in $\FF$ 
 is at most $(m+1)s + t + f(t)$.
Using $(m+1)+s = n-t-1$, $n\geq t+5$, and the notation $\delta(t):=f(t) -\frac{1}{4}(t-1)^2$ with the convention $\delta(0) = -\frac{1}{4}$, we obtain
\begin{equation*}
    \begin{split}
   |\FF|- \frac{1}{4}(n-1)^2 & \leq - \frac{1}{4}(n-1)^2 + \frac{1}{4}(n-t-1)^2 + t+ \frac{1}{4}(t-1)^2 +\delta(t)\\
   &=  \frac{1}{4}(2t^2+4t+1 -2nt)+ \delta(t) \leq  \frac{1}{4}(1-6t)+ \delta(t).
    \end{split}
\end{equation*}
\noindent
Note that $\delta(t) = -\frac{1}{4}$ for even $t$, and for odd $t$ we have $\delta(1)=\delta(3)=0$ and $\delta(t)=1$ when $t \ge 5$. We know that $|\FF| - \frac{1}{4}(n-1)^2 \ge \delta(n) \ge -\frac{1}{4}$, whereas the upper bound $\frac{1}{4}(1-6t)+ \delta(t)$ is at most $-\frac{5}{4}$ for $t \ge 1$. So we must have $t=0$ and $n$ even. In this case, as we have seen, $\FF$ consists of all triples $aby$ with $a \in \{x\} \cup M$, $b \in U$. Thus $\FF$ is the construction based on the Tur\'an graph with the vertex $y$ common to all triples, i.e., $\FF=\FF^*_n$,
completing the proof of $\FF\in \FFF_{(n)}$ in Case~5.
\qed

\subsection{{Proof of Theorem~\ref{thm:AC}}: the end}

Recall that we partitioned $[n]$ into the connected components $C_1, ... , C_r$ of $G_4$ of sizes $p_1, \dots, p_r$, $p_i\geq 2$ for all $i$.
We may also suppose that each pendant $G_4$ edge is a $G_{42}$ edge.

\subsubsection{Case 6: The opposite of Case 5}\label{sss:971}

This is the last case, and we will show that it leads to a contradiction.
We suppose that for each component $C$ of size $p$ either $G_4^C$ is not a tree, i.e., $p-|E(G_4^C)|\leq 0$, call this Case 6/1,
 or it is not near optimal, i.e.,  $ |E(G_1^C)|-\frac{p-2}{2}< |E(G_3^C)|$ (Case 6/2).
Note that in both cases
\begin{equation}\label{eq:98}
    \lceil \frac{p}{2}\rceil - |E(G_4^C)| +|E(G_1^C)|- |E(G_3^C)| \leq 0.
\end{equation}

Indeed, Lemma~\ref{lem:G13} states  $ |E(G_1^C)| - \lfloor \frac{p}{2}\rfloor \leq   |E(G_3^C)|$.
Adding this to $p-|E(G_4^C)|\leq 0$ we obtain~\eqref{eq:98} in Case 6/1. In Case 6/2 we add
 $ |E(G_1^C)|-\frac{p-3}{2}\leq |E(G_3^C)|$ to the obvious $p-1-|E(G_4^C)|\leq 0$ and get
 $\frac{p+1}{2} - |E(G_4^C)| +|E(G_1^C)|- |E(G_3^C)| \leq 0$.

Hence in the sum in the last line of~\eqref{eq:96core} each term $\sigma(C_i)$ is  non-positive. Even more, in Case 6/1 we have $\sigma(C_i)\leq -\frac{1}{2}$
when $p_i$ is odd, and $\sigma(C_i)$ is negative in Case 6/2 as well.
Therefore in Case 6 we get $|\FF|\leq \frac{1}{4}(n^2-2n)\leq f(n)$.
Here equality must hold, so to avoid negative terms each $p_i$ should be even and Case 6/2 is also excluded.

So from now on, we may suppose 6/1 holds with equality in~\eqref{eq:98} for each $i$.
So each $p_i$ is even, $p_i=|E(G_4^{C_i})|$, and equality holds in Lemma~\ref{lem:G13}, i.e., (L1) holds.
Thus $C_i$ contains a perfect matching 
 of $G_1$ edges.
Then Claim~\ref{cl:96} implies that $G_4^{C_i}$ has no pendant edge, each degree is at least $2$.
So $G_4^{C_i}$ should be $2$-regular, connected, i.e., each $G_4^{C_i}$ is an even cycle, $p_i\geq 4$.

We are going to show that each $p_i$ is exactly $4$. Let the vertex set of $C$ be $[p]$ with $G_4^C$ edges
 $12, 23, ..., (p-1)p$, and $p1$.
The $G^C_1$ edges are diagonals of length 2 in $G_4^C$. They also contain a matching of size $\frac{p}{2}$, and this is only possible if there are crossing edges like $13$ and $24$ with
 $123, 234\in \FF_1$.

Suppose that there exists a triple containing $14$, say $14z\in \FF$.
Since $13$ and $24$ are covered exactly once by the triples of $\FF$ we have $z\notin \{ 1,2,3,4\}$.
The pair $12$ is not a $G_{41}$ edge (because $123$ does not have two own pairs) so it is a $G_{42}$ edge, there exists a triple $12w\in \FF$, $w\neq 3$, containing it.
The triples $12w, 14z, 234$ form an $\AAA$ configuration unless $w\in \{ 4, z\}$.
To avoid $\CC$ we have $w\neq 4$, so $w=z$ and $12z\in \FF$.
One can have the same argument starting with the pair $34$ in place of $12$ and obtain $34z\in \FF$. Then we got a five-ring with vertex set $\{ 1,2,3,4,z\}$, a contradiction, because $12$, etc., are $G_4$ edges.

We conclude that the pair $14$ is uncovered.
If it is a $G_3$ edge then there are two $G_1$ pairs $w1$ and $w4$, both diagonals of length $2$. But this is impossible since $p$ is even and the arc $1234$ is of odd length. So $14$ is a $G_{40}$ edge. Thus $1234$ form a $4$-cycle in $G_4$, $p=4$.
This must be true for each component, so $G_4$ is a vertex disjoint union of $4$-cycles such that each contains two triples and a $G_{40}$ pair.
So $G_1$ is a perfect matching and there is no  $G_3$ edge at all.

Consider any $4$-cycle $C$ in $G_4$, say with $123, 234 \in \FF$, $14\in E(G_{40})$ and $12, 23, 34\in E(G_{42})$.
We have a triple $12z\in \FF$ with $z\neq 3$, and (to avoid $\CC$) $z\neq 4$.
Every pair $4z$ with $z\notin C$ is contained in a triple, so we have some triple $4zw\in \FF$.
The triples $12z, 234, 4zw$ form an $\AAA$ configuration unless $w\in \{ 1,2,3\}$. The case $w=1$ is not allowed because $14\in E(G_{40})$, and the case $w=2$
is not allowed because $24\in E(G_1)$. So $w=3$, $34z\in \FF$.
The triples $12z, 34z$ have no $G_1$ pair, hence their pairs $1z, 2z, 3z, 4z$ are $G_2$ pairs. The pair $23$ cannot appear in any triple in $\FF$ except $123, 234$. Indeed, $23z$ is excluded because $2z, 3z$ are own pairs of other triples, and $23w$ for $w \notin C \cup \{z\}$ creates the $\AAA$ configuration $12z, 34z, 23w$. We conclude that the set of triples in $\FF$ that meet $C$ in exactly two vertices is of the form $\{ 12z, 34z: z \in Z_C \}$ where $Z_C$ is a non-empty set disjoint from $C$.

Next, we discuss five-rings. We claim that they must be vertex disjoint. Indeed, any two of them can meet in at most one vertex. If two of them on vertex sets $R_1, R_2$ share a vertex $x$, then in order to avoid an $\AAA$ configuration all $16$ pairs $ab$, $a \in R_1 \setminus \{x\}$, $b \in R_2 \setminus \{x\}$ must be uncovered. However, the only uncovered pairs are $G_{40}$ edges and they are pairwise disjoint.

Now, for any two distinct $4$-cycles $C_i, C_j$ in $G_4$, we write $C_i \to C_j$ if $Z_{C_i} \cap C_j \ne \emptyset$. As shown above, in this case all pairs in $C_i \times (Z_{C_i} \cap C_j)$ are $G_2$ edges. For any $i \ne j$, exactly one of $C_i \to C_j$ and $C_j \to C_i$ holds. Indeed, if both held then pairs in $(Z_{C_j} \cap C_i) \times (Z_{C_i} \cap C_j)$ would be $G_2$ edges contained in two triples in $\FF$. If neither held then all pairs in $C_i \times C_j$ would have to be $G_R$ edges (nothing else is allowed), which is impossible because adjacent $G_R$ edges must come from the same five-ring (as five-rings are disjoint).

We have thus defined a tournament $T$ on $\{C_1,\ldots,C_r\}$. As $Z_{C_i} \ne \emptyset$ for every $i$, every $C_i$ has at least one out-neighbor in $T$, so $T$ cannot be transitive. But then it contains a directed triangle, which gives rise to an $\AAA$ configuration. This contradiction completes the proof.
\qed

\section{Excluding $\AAA$, $\BB$ and $\DD$}
Recall that $\ex(n,\AAA\BB\CC\DD)=\lfloor \frac{1}{8}n^2 \rfloor$ for all $n\geq 1$ and in Theorem~\ref{thm:all} in Section~\ref{sec:all} we have described $\FF_{n,k}$, the extremal families.
In this section we show that for large enough $n$, excluding $\AAA$, $\BB$ and $\DD$ has the same effect on the size of a triple system as excluding all triangles.
For finitely many cases $\ex(n,\AAA\BB\DD)$ can exceed $\lfloor \frac{1}{8}n^2 \rfloor$, e.g., for $n=4$
 the complete triple system $K^3_4:= \binom{[4]}{3}$ is $\AAA\BB\DD$-free, while $\lfloor \frac{1}{8}n^2 \rfloor=2$.
For $n=7$ two $K^3_4$ sharing a vertex give $\ex(7,\AAA\BB\DD)\geq 8$ (and here equality holds)
while $\lfloor \frac{1}{8}n^2 \rfloor=6$.
For $n=16$ the family of $8$ $K_4^3$ occupying the rows and columns of a $4\times 4$ square lattice is $\AAA\BB\DD$-free and has a completely different structure than  $\FF_{n,k}$. Denote this triple system by $\QQ^{4\times 4}$, and note that $|\QQ^{4 \times 4}| = \lfloor \frac{1}{8}n^2 \rfloor = 32$.

\begin{thm} \label{thm:ABD}
\[ \ex(n,\AAA\BB\DD) = \lfloor \frac{1}{8}n^2 \rfloor \, \text{  for } n \ge 8. \]
Moreover, for $n\geq 11$  the families  $\FF_{n,k}$ described in Section~\ref{sec:all} are the only extremal systems, except that
 $\QQ^{4\times 4}$ is also extremal in the case $n=16$.
\end{thm}

For $n=8,9,10$ here is the complete list of extremal families.\\
\noindent $n=8$: $\FF_{8,2}= \{ 12x: 5\leq x \leq 8\} \cup \{ 34x: 5\leq x \leq 8\}$, \\
\noindent ${}\quad\quad$\quad  $\QQ_8^{2}=$ two $K^3_4$ sharing a vertex (this triple system has an isolated vertex),\\
\noindent ${}\quad\quad$\quad  $\QQ_8^{3}=$ two vertex disjoint $K^3_4$.\\
\noindent $n=9$: $\FF_{9,2}= \{ 12x: 5\leq x \leq 9\} \cup \{ 34x: 5\leq x \leq 9\}$, \\
\noindent${}\quad\quad$\quad  $\QQ_9^{2}=$ two $K^3_4$ on 1234 and 1567 with two additional triples $289$, $589$.\\
\noindent $n=10$: $\FF_{10,2}= \{ 12x: 5\leq x \leq 10\} \cup \{ 34x: 5\leq x \leq 10\}$, \\
\noindent ${}\quad\quad$\quad  $\FF_{10,3}= \{ 12x: 7\leq x \leq 10\} \cup \{ 34x: 7\leq x \leq 10\} \cup \{ 56x: 7\leq x \leq 10\}$, \\
\noindent ${}\quad\quad$\quad  $\QQ_{10}^{3}=$ three $K^3_4$ sharing a vertex,\\
\noindent ${}\quad\quad$\quad  $\QQ_{10}^{4}=$ three $K^3_4$ forming a path, $\{ 1234, 4567, 789X\}$.


\subsection{Non-uniform triangle-free hypergraphs}\label{ss:101}

Gy\H{o}ri~\cite{GyE} proved 
his upper bound 
in the following more general form.
Let $\HH$ be a triangle-free multi-hypergraph on $n \ge 100$ vertices. Then
\begin{equation} \label{eqn:Gy}
\sum_{E \in \HH} \bigl( |E| - 2 \bigr) \le \lfloor \frac{1}{8}n^2 \rfloor.
\end{equation}
In this formulation, $\HH$ may contain edges of different sizes, and may contain multiple copies of the same set of vertices. Such multiple copies are considered distinct for the purpose of forming a triangle, and are counted with multiplicity in the summation in~(\ref{eqn:Gy}).

Given an $\AAA\BB\DD$-free triple system $\FF$ one can prove $|\FF|\leq \lfloor \frac{1}{8}n^2 \rfloor$ for $n \ge 100$ using~\eqref{eqn:Gy} as follows. Let $\QQ:=\{ Q_1,\ldots,Q_q\}$ be all the vertex sets of $\CC$ configurations in $\FF$, i.e., $|Q_i|=4$ and at least three triples within $Q_i$ are in $\FF$. Note that distinct $Q_i$ can share at most one vertex, otherwise we get a $\DD$ configuration. Let $\HH$ be the multi-hypergraph having the following edges:
 two copies of each $Q_i$, $i=1,\ldots,q$ and  one copy of each $F \in \FF$ such that $F \not\subseteq Q_i$, $i=1,\ldots,q$.
Note that
\begin{equation} \label{eqn:101}
  |\FF| \leq 4|\QQ| + |\FF \setminus (\partial\QQ)| =  \sum_{E \in \HH} \bigl( |E| - 2 \bigr).
\end{equation}
Indeed, the inequality holds because each $Q_i$ contains at most $4$ triples in $\FF$. The equality holds because each $Q_i$ contributes $\bigl( 4-2 \bigr) + \bigl( 4-2 \bigr) = 4$ to the sum, and every other triple in $\FF$ contributes $3-2=1$. Upon checking that $\HH$ is triangle-free, (\ref{eqn:Gy}) implies  the desired bound $\lfloor \frac{1}{8}n^2 \rfloor$ for the right-hand side of~\eqref{eqn:101}.

The above proof does not yield the extremal systems and only works for $n\geq 100$.

\begin{defi}\label{def:Q4}
A (simple) hypergraph $\HH$ is called a $Q4$-{\em family} on $[n]$ if there is a partition $\HH=\QQ\cup \TT\cup\EE$ such that
\begin{itemize}
  \item[--]  $\QQ\subseteq {[n]\choose 4}$, $\TT\subseteq {[n]\choose 3}$, $\EE\subseteq {[n]\choose 2}$,
  \item[--] $\HH$ is triangle-free, and
  \item[--] $|H\cap H'|\leq 1$ for any two distinct members $H,H'\in \HH$.
\end{itemize}
\end{defi}


\begin{lem} \label{lem:21}
Let $\HH$ be a $Q4$-{\em family} on $[n]$ with $\QQ\neq \emptyset$ and suppose $n\geq 15$.  Then
\begin{equation} \label{eq:21}
  8|\QQ| + \frac{7}{2}|\TT|+|\EE|< 2\lfloor \frac{1}{8}n^2 \rfloor
\end{equation}
except in the case when $\TT = \EE = \emptyset$ and $\partial\QQ$ is the triple system $\QQ^{4 \times 4}$.
\end{lem}

This is a simplified version of an inequality from~\cite{GyE} where Gy\H ori considered $8|\QQ|+4|\TT|+|\EE|$.
Although~\eqref{eq:21} does not seem to imply~\eqref{eqn:Gy}, and it does not apply to multi-hypergraphs, its proof is simpler and valid for $15 \leq n < 100$, too.

\begin{proof}
Let $\sigma(\HH):=  8|\QQ| + \frac{7}{2}|\TT|+|\EE|$,
and
  $\sigma_3(n):= \max \{\sigma(\HH): \HH$  is a $Q4$-family on $[n]$ with $\QQ=\emptyset\}$.
Since a complete bipartite graph can be considered as a $Q4$-family we have $
\sigma_3(n)\geq
  \lfloor \frac{1}{4}n^2 \rfloor$.
As a first step we prove that for all $n\geq 0$
\begin{equation} \label{eq:22}
   \frac{7}{2}|\TT|+|\EE| \leq   \frac{1}{4}(n^2+5).
\end{equation}
\noindent
Indeed, $\sigma_3(0)=\sigma_3(1)=0$, $\sigma_3(2)=1=\frac{1}{4}2^2$,  $\sigma_3(3)=\frac{7}{2}=\frac{1}{4}(3^2+5)$. By inspection we can see that
 $\sigma_3(n)$ is maximized for $n=4,5,6$  when $\HH$ consists of a single
 triple and an edge, two triples sharing a vertex, or two disjoint triples joined by three disjoint pairs.
This gives $\sigma_3(4)=\frac{9}{2}=\frac{1}{4}(4^2+2)$,  $\sigma_3(5)=7=\frac{1}{4}(5^2+3)$, and  $\sigma_3(6)=10=\frac{1}{4}(6^2+4)$.

For $n\geq 7$ let $T_1, \dots, T_\ell$ be a maximal family of pairwise disjoint triples from $\TT$, let $L= \cup T_i$, $|L|=3\ell$, and $\overline{L}:= [n]\setminus L$. The case $\ell=0$ is obvious, so we may suppose that $\ell\geq 1$.
Make a tally of the pairs in $G:=(\partial\TT) \cup \EE$.
Every vertex $x\in [n]\setminus T_i$ has at most one $G$-neighbor in $T_i$, so $G$ can have at most
$3\ell + 3{\ell \choose 2}$ edges in $L$, $\ell(n-3\ell)$ edges joining $L$ and $\overline{L}$. Since $G$ is triangle-free on $\overline{L}$ it can have at most $\frac{1}{4}(n-3\ell)^2$ edges in it.
Define weights $w(e)$ on the edges of $G$ as follows. The weight $w(e) = \frac{7}{6}$ if $e$ lies in $L$,
 $w(e)=\frac{5}{4}$ if $e$ joins  $L$ and $\overline{L}$, and finally $w(e)=1$ for edges in $\overline{L}$.
The sum of the weights of the three pairs in $T\in \TT$ is $3\times \frac{7}{6}$ if $T\subseteq L$,
 it is $\frac{7}{6}+2\times \frac{5}{4}$ if $|T\cap L|=2$, it is
 $2\times \frac{5}{4} +1$ if $|T\cap L|=1$, and there is no $T\in \TT$ such that $|T\cap L|=0$.
So every triangle $T$ in $G$ belongs to $\TT$ and receives at least $\frac{7}{2}$ total weight.
Hence
  \begin{equation}\label{eq:222}
    \begin{split}
     \frac{7}{2}|\TT|+|\EE| & \leq \frac{7}{6}\left(3\ell + 3{\ell \choose 2}\right)+
 \frac{5}{4}\ell(n-3\ell)+ \frac{1}{4}(n-3\ell)^2  \\
    & =  \frac{1}{4}\left( n^2 + \ell^2 -n\ell +7\ell \right) .
    \end{split}
    \end{equation}
The  polynomial $p_n(\ell):= \ell^2 -n\ell +7\ell $ attains its maximum over $0\leq \ell \leq \lfloor \frac{n}{3}\rfloor$ at $\ell=0$ or $\ell=\lfloor \frac{n}{3}\rfloor $, so it is at most $4$ for $n=7$, at most $2$ for $n=8$, at most $3$ for $n=9$, and at most $0$ for $n\geq 10$. This completes the proof of inequality~(\ref{eq:22}).

\vspace{10pt}

Next, consider the case $|\QQ|\geq 1$ and let $Q_1, \dots, Q_k$ be a maximal family of pairwise disjoint quadruples from $\QQ$, let $K= \cup Q_i$, $|K|=4k$, and $\overline{K}:= [n]\setminus K$.
Define the hypergraph $\HH_3$ on the vertex set $\overline K$ as
\begin{itemize}
  \item[--]  the triples  of $\TT$ and the pairs of $\EE$ contained in $\overline K$, and
  \item[--]  the triples of the form $Q\cap\overline K $ when $Q\in \QQ$, $|Q\cap K|=1$, and
  \item[--]  the pairs of the form $Q\cap\overline K $ when $Q\in \QQ$, $|Q\cap K|=2$, and
  \item[--]  the pairs of the form $T\cap\overline K $ when $T\in \TT$, $|T\cap K|=1$.
\end{itemize}
This is a $Q4$-family without any quadruple, so~\eqref{eq:22} gives
$$
  \sigma(\HH_3)\leq \frac{1}{4}((n-4k)^2+5).
   $$
Let $\partial_2 \HH$ be the set of pairs that are contained in some edge of $\HH$.
Define the graph $G$ as the set of pairs in $\partial_2 \HH$ meeting $K$.
As before,   $G$ can have at most
$6k + 4{k \choose 2}$ edges in $K$ and $k(n-4k)$ edges joining $K$ and $\overline{K}$.
Define weights $w(e)$ on the edges of $G$ as follows. The weight $w(e) = \frac{4}{3}$ if $e$ lies in $K$,
 $w(e)=\frac{3}{2}$ if $e$ joins  $K$ to $\overline{K}$.
Define weights on the members of $\HH_3$ as needed,
$w(S)=\frac{7}{2}$ for triples in $\HH_3$  and $w(e)=1$ for pairs in $\HH_3$.

It is easy to check that the total weight in each member of $Q\in \QQ$ is at least 8,
  it is at least $\frac{7}{2}$ for the members of $\TT$, and it is at least 1 for pairs in $\EE$. Indeed,
if $Q\subseteq K$ then it gets $6\times \frac{4}{3}=8$, if $|Q\cap K|=3$ then it gets $3\times \frac{4}{3}+3\times \frac{3}{2}$,
 if $|Q\cap K|=2$ then it gets $\frac{4}{3}+4\times \frac{3}{2}+1$, and if  $|Q\cap K|=1$ then it gets $3\times \frac{3}{2}+\frac{7}{2}$.
If  $T\in \TT$ and $T\subseteq K$ then it gets $3\times \frac{4}{3}$, if $|T\cap K|=2$ then it gets $\frac{4}{3}+2\times \frac{3}{2}$,  if $|T\cap K|=1$ then it gets $2\times \frac{3}{2}+1$, and if $T\cap K=\emptyset$ then
its weight is $\frac{7}{2}$.
Finally, the weight of the (relevant) pairs is at least $1$.
Hence
  \begin{eqnarray}
   8|\QQ| + \frac{7}{2}|\TT|+|\EE| & \leq & \sum_{e \in G} w(e) + \sigma_3(n-4k) \label{eq:223}\\
    & \leq &
      \frac{4}{3}\left(6k + 4{k\choose 2}\right)+
 \frac{3}{2}k(n-4k)+ \frac{1}{4}\left((n-4k)^2 +5\right)\label{eq:224} \\
    & = & \frac{1}{4}n^2 + \frac{1}{12}\left( 8k^2 -6k n +64k +{15} \right) .\notag
    \end{eqnarray}
Here $2k(4k-n + 32-2n)+{15}< -12$ for $n\geq 18$ (and $1\leq k\leq \lfloor \frac{n}{4} \rfloor$),
 it is less than $-3$ for $n=17$,  and it is negative for $n=16$ and $k=1,2,3$.

For the case $n=16$ and $k=4$ we have $K=V(G)$ and $G$ has $24$ edges in $Q_1, \dots, Q_4$ and at most $4\times{4\choose 2}=24$ further edges.
Recall that the weight of each edge is $\frac{4}{3}$. So every $Q\in \QQ$ gets weight 8, $T\in \TT$ gets weight 4, and a pair $E\in \EE$ has weight $\frac{4}{3}$. So $\sigma (\HH)\leq 48\times \frac{4}{3}=64$, as claimed.
Here equality can hold only if $\TT$ and $\EE$ are both empty, and $|\QQ|=8$. Then $\QQ$ must have the lattice structure.

Finally, in the case $n=15$ we use~\eqref{eq:223} directly.
By~\eqref{eq:222} we have $\sigma_3(15-4)\leq \frac{1}{4}11^2$ so in the case $(n,k)=(15,1)$
 instead of~\eqref{eq:224} we have $\sigma(\HH) \leq \frac{1}{4}n^2 + \frac{1}{12}\left( 8k^2 -6k n +64k \right)$ which is less than $2\lfloor \frac{1}{8}n^2 \rfloor$.
Similarly, in the case $(n,k)=(15,2)$ using $\sigma_3(15-2\times 4)\leq \frac{1}{4}(7^2+4)$ leads to
 $\sigma(\HH) < 2\lfloor \frac{1}{8}n^2 \rfloor$.

In the case $(n,k)=(15,3)$ we define another weight function on the edges of $\partial_2 \HH$. We assign weight $\frac{4}{3}$ to each pair.
With this weighting every member of $\QQ$ receives a total weight $8$, a triple in $\TT$ gets $4$ and a pair from $\EE$ gets $\frac{4}{3}$.
The graph $\partial_2 \HH$ has $18$ edges covered by $Q_1, Q_2$, and $Q_3$, there are at most $12$ edges between these three quadruples, there are at most
 $3\times 3$ edges from $\overline K$ to $K$, and at most $\binom{|{\overline K}|}{2}=3$ edges in $\overline K$. Hence $|\partial_2 \HH|\leq 42$ and
 $\sigma(\HH)\leq 56 = 2\lfloor \frac{1}{8}n^2 \rfloor$. However, here equality can hold only if all the $42$ pairs of  $\partial_2 \HH$ belong to members of $\QQ$. But it is not possible to add to $Q_1, Q_2, Q_3$ four more quadruples that form with them a $Q4$-family.
This completes the proof of the lemma. \end{proof}


\subsection{Proof of Theorem~\ref{thm:ABD}}

Let $\FF$ be an $\AAA\BB\DD$-free triple system on $[n]$ with $|\FF|\geq \lfloor \frac{1}{8}n^2 \rfloor$, $n\geq 8$.
Our aim is to prove that  $|\FF|=\lfloor \frac{1}{8}n^2 \rfloor$ and $\FF$ is among the families prescribed earlier.
We start as in the beginning of Subsection~\ref{ss:101}.

Let $\QQ:=\{ Q_1,\ldots,Q_q\}$ be all the vertex sets of $\CC$ configurations in $\FF$, i.e., $|Q_i|=4$ and at least three triples within $Q_i$ are in $\FF$.
Define $\FF_3$ as the set of triples from $\FF$ having three own pairs, and let $\FF_2:= \FF \setminus (\partial \QQ \cup \FF_3)$.
Each member of $\FF_2$ has exactly two own pairs, the set of these are denoted by $\EE$.
Then the family $\HH= \QQ \cup \FF_3 \cup \EE$ is a $Q4$-family.

If $\QQ=\emptyset$ then $\FF$ is $\AAA\BB\CC\DD$-free and Theorem~\ref{thm:all} can be applied. So from now on, we may suppose that $q\geq 1$.
In case of $n\geq 15$ we can apply Lemma~\ref{lem:21}.
We get
$$
  2|\FF|\leq 8|\QQ|+ 2|\FF_3|+ |\EE| \leq \sigma(\HH)< 2\lfloor  \frac{1}{8}n^2 \rfloor$$
 unless $\FF=\QQ^{4\times 4}$.
So the proof is complete for $n\geq 15$.

From now on we may suppose that $n\leq 14$ (and $q\geq 1$).
These are finitely many cases so one can check them with a computer.
For completeness we sketch a case by case check in the next subsection.
\qed

\subsubsection{The case $8\leq n\leq 14$}
Starting with $Q_1=1234$ and adding new quadruples one by one it can be seen that there are 16 possibilities for $\QQ$.
The list is presented below (also see Figure~\ref{quadtypes}). $\kappa$ denotes $|\cup \QQ|$, $U:=\cup \QQ$, and $i'$ denotes the element $10+i$.

\begin{figure}
\centering
\includegraphics[width=1\textwidth]{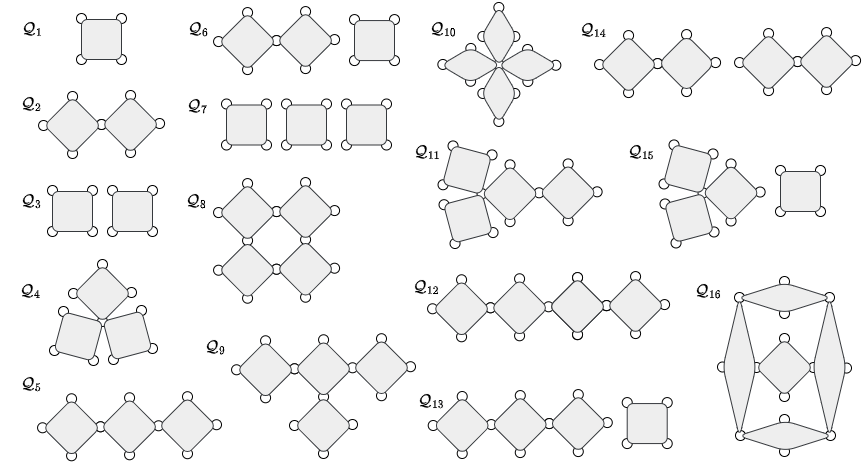}
\caption{\label{quadtypes}The 16 possibilities for $\QQ$ on 14 vertices}
\end{figure}

\begin{enumerate}
\item $q=1$, $\kappa=4$, $\QQ_1=\{ 1234\}$
\item $q=2$, $\kappa=7$, $\QQ_2=\{ 1234, 1567\}$
\item $q=2$, $\kappa=8$, $\QQ_3=\{ 1234, 5678\}$
\item $q=3$, $\kappa=10$, $\QQ_4=\{ 1234, 1567, 1890'\}$
\item $q=3$, $\kappa=10$, $\QQ_5=\{ 1234, 4567, 7890'\}$
\item $q=3$, $\kappa=11$, $\QQ_6=\{ 1234, 1567, 890'1'\}$
\item $q=3$, $\kappa=12$, $\QQ_7=\{ 1234, 5678, 90'1'2'\}$
\item $q=4$, $\kappa=12$, $\QQ_8=\{ 1234, 5678, 1590', 261'2'\}$
\item $q=4$, $\kappa=13$, $\QQ_9=\{ 1234, 5678, 90'1'2', 1593'\}$
\item $q=4$, $\kappa=13$, $\QQ_{10}=\{ 1234, 1567, 1890', 11'2'3'\}$
\item $q=4$, $\kappa=13$, $\QQ_{11}=\{ 1234, 1567, 1890', 81'2'3'\}$
\item $q=4$, $\kappa=13$, $\QQ_{12}=\{ 1234, 4567, 7890', 0'1'2'3'\}$
\item $q=4$, $\kappa=14$, $\QQ_{13}=\{ 1234, 5678, 890'1', 1'2'3'4'\}$
\item $q=4$, $\kappa=14$, $\QQ_{14}=\{ 1234, 1567, 890'1', 82'3'4'\}$
\item $q=4$, $\kappa=14$, $\QQ_{15}=\{ 1234, 5678, 590'1', 52'3'4'\}$
\item $q=5$, $\kappa=14$, $\QQ_{16}=\{ 1234, 5678, 90'1'2', 1593',260'4'\}$.
\end{enumerate}

In all cases we define a graph $G$ by selecting two own pairs from each triple in $\FF \setminus (\partial\QQ)$.
This $G$ is triangle-free.
Then we find an upper bound $\alpha$ for the possible number of edges of $G$ in $U$, a bound $\beta$
 for the number of $G$-edges joining $U$ to $[n]\setminus U$, and a bound $\gamma$ for the number of $G$-edges in $[n]\setminus U$.
In particular, we have $\gamma \leq \lfloor \frac{1}{4}(n-\kappa)^2 \rfloor$.
Then
\begin{equation}\label{eq10.6}
   |\FF| \leq 4|\QQ|+ \frac{1}{2}|E(G)|\leq 4q + \lfloor \frac{1}{2}(\alpha + \beta+ \gamma)\rfloor.
  \end{equation}
In almost all cases~\eqref{eq10.6} immediately gives $|\FF|< \lfloor  \frac{1}{8}n^2 \rfloor $.

\noindent
--- Case of $\QQ_1$. \enskip
$\alpha =0$, $\beta\leq n-4$, $\gamma\leq \frac{1}{4}(n-4)^2$, so
$|\FF|\leq 4+ \frac{1}{2}(\alpha + \beta+ \gamma)\leq \frac{1}{8}(n^2-4n+32)$. This is less than
$\lfloor  \frac{1}{8}n^2 \rfloor $ for $n\geq 9$.
In case of $n=8$ equality must hold, $G$ is a $4$-cycle on $[8]\setminus Q_1$ plus four edges to $Q_1$. But this cannot be realized by triples.

\noindent
--- Case of $\QQ_2$. \enskip
$\alpha =0$, $\beta\leq 2(n-7)$, $\gamma\leq \lfloor \frac{1}{4}(n-7)^2\rfloor$.  So
$|\FF|< \lfloor \frac{1}{8}n^2 \rfloor $ for $n\geq 11$ by~\eqref{eq10.6}.
In case of $n=10$ we have that $\gamma=2$ and $\beta=6$, but this cannot be realized.
In case of $n=9$ we obtain $\QQ_9^2$.
The case of $n=8$ is obvious, $\FF=\partial \QQ$ ($= \QQ_8^2$).

\noindent
--- Case of $\QQ_3$. \enskip
For $n\geq 11$ we have $\alpha\leq 4$, $\beta\leq 2(n-8)$, so $|\FF|\leq \frac{1}{8}(n^2-8n+80)< \lfloor \frac{1}{8}n^2 \rfloor$.
In case of $n=8$ we have $\FF=\partial \QQ$ ($= \QQ_8^3$).
If $n=9$ then only one more triple can be added to $\QQ_8^3$, and in case of $n=10$ at most two more can be added. These are too few, so we are done.

\noindent
--- Case of $\QQ_4$. \enskip
Here $\alpha=0$, $\beta\leq 3(n-10)$ so~\eqref{eq10.6} gives $|\FF|< \lfloor \frac{1}{8}n^2 \rfloor$ for $n\geq 11$.
In case of $n=10$ we have $\FF=\partial \QQ$ ($= \QQ_{10}^3$).

\noindent
--- Case of $\QQ_5$. \enskip
Here $\alpha=3$, $\beta\leq 3(n-10)$ so~\eqref{eq10.6} gives $|\FF|
 \leq \frac{1}{8}(n^2-8n+88)< \lfloor \frac{1}{8}n^2 \rfloor$ for $n\geq 12$.
In case of $n=10$ we have $\FF=\partial \QQ$ ($= \QQ_{10}^4$).
In case of $n=11$ one can add at most one new triple to  $\QQ_5$, so we have $|\FF|\leq 13$.

\noindent
--- Case of $\QQ_6$. \enskip
Here $\alpha=7$, $\beta\leq 3(n-11)$ so~\eqref{eq10.6} gives $|\FF|
\leq \frac{1}{8}(n^2-10n+113) < \lfloor \frac{1}{8}n^2 \rfloor$ for $n\geq 12$.
In case of $n=11$ we have $\FF=\partial \QQ$ so $|\FF|\leq 12$.

\noindent
--- Case of $\QQ_7$. \enskip
Here $\alpha=12$, $\beta\leq 3(n-12)$ so~\eqref{eq10.6} gives $|\FF|
 \leq \frac{1}{8}(n^2-12n+144)< \lfloor \frac{1}{8}n^2 \rfloor$ for $n\geq 13$.
In case of $n=12$ one can add at most four disjoint new triples to  $\QQ_7$, so we have $|\FF|\leq 16$.

\noindent
--- Case of $\QQ_8$. \enskip
Here $\alpha=4$, $\beta\leq 4(n-12)$ so~\eqref{eq10.6} gives $|\FF|
 \leq \frac{1}{8}(n^2-8n+96)< \lfloor \frac{1}{8}n^2 \rfloor$ for $n\geq 13$.
In case of $n=12$ one cannot add any new triple to  $\QQ_8$, so we have $|\FF|\leq 16$.

In the rest of the cases $n\in \{13, 14\}$ and $|U|\geq 13$.
Let $\eta$ (or $\eta_{n}$) be the size of $\FF\setminus (\partial \QQ)$.
Since $|\QQ|$ is at least 4 we have more restrictions, it is easier to give an upper bound for $\eta$.
In the cases $\QQ_9$--$\QQ_{15}$ we have to show that $\eta_{13}< 21-16=5$, and $\eta_{14}< 8$.
In the case $\QQ_{16}$ we need $\eta_{14}< 4$.

\noindent
--- Case of $\QQ_9$. \enskip
No more triples meeting $1593'$, so $\eta_{13}\leq 3$ and $\eta_{14}\leq 6$.

\noindent
--- Case of $\QQ_{10}$. \enskip
 Obviously $\eta_{13}=0$ and $\eta_{14}=0$.

\noindent
--- Case of $\QQ_{11}$. \enskip
 $\eta_{13}=0$ and $\eta_{14} \le 2$.

\noindent
--- Case of $\QQ_{12}$. \enskip
 $\eta_{13}=0$ and $\eta_{14}\leq {4 \choose 2}$.

\noindent
--- Case of $\QQ_{13}$. \enskip
No more triples meeting $890'1'$, so $\eta\leq 3$.

\noindent
--- Case of $\QQ_{14}$ and $\QQ_{15}$. \enskip
$\eta =0$.

\noindent
--- Case of $\QQ_{16}$. \enskip
$\eta \le 2$.
\qed

\section{Excluding $\BB$, $\CC$ and $\DD$} \label{sec:BCD}
For $\BB\CC\DD$-free families we know the upper bound \[ \ex(n,\BB\CC\DD) \le \ex(n,\CC\DD) = \lfloor \frac{1}{4}(n-1)^2 \rfloor.\]
The extremal construction for $\ex(n,\CC\DD)$ -- of the form $\FF_T$ for some tournament $T$ -- does contain $\BB$ configurations, so this bound cannot be exactly matched. But we do have an asymptotically matching construction.

\begin{thm} \label{thm:BCD}
\[\ex(n,\BB\CC\DD) = \frac{1}{4}n^2 \bigl( 1 - o(1) \bigr).\]
\end{thm}

Our construction uses Wilson's well-known decomposition theorem, which we recall now.

\begin{thm}[Wilson~\cite{Wi}] \label{thm:W}
For every fixed graph $H$ there exists $n_0 = n_0(H)$ such that the following holds for all $n \ge n_0$: \\There exists a decomposition of the edge set $E(K_n)$ of the complete graph on $n$ vertices into copies of $H$ if and only if the number of edges of $H$ divides ${n \choose 2}$, and the greatest common divisor of the vertex degrees in $H$ divides $n-1$.
\end{thm}

\noindent \textbf{Proof of Theorem~\ref{thm:BCD}.}\quad
For every integer $m \ge 2$ we construct a graph $H_m$ to which we will apply Theorem~\ref{thm:W}. It has $m^2$ vertices, partitioned into ${m \choose 2}$ pairs $A_1,\ldots,A_{{m \choose 2}}$ and a set $B$ of size $m$. The graph induced on each $A_i \cup B$, $1 \le i \le {m \choose 2}$, is complete, and these are all the edges of $H_m$. Thus
\begin{itemize}
    \item $|E(H_m)| = {m \choose 2}(2m+1) + {m \choose 2} = {m \choose 2} (2m+2) = m(m-1)(m+1)$.
    \item $\gcd\bigl( \{\deg_{H_m}(v)\colon v \in V(H_m)\}\bigr) = \gcd(m+1,m^2-1) = m+1$.
\end{itemize}
Let $n \ge n_0(H_m)$ be such that $2m(m-1)(m+1)$ divides $n-1$. Then both divisibility conditions in Theorem~\ref{thm:W} are satisfied, so we can find a decomposition of $E(K_n)$ into $\frac{1}{m(m-1)(m+1)}{n \choose 2}$ copies of $H_m$.

In each copy of $H_m$ in the decomposition, we place ${m \choose 2}m$ triples of the form $A_i \cup \{b\}$ where $1 \le i \le {m \choose 2}$ and $b \in B$. Within each copy, this triple system is triangle-free: it actually follows the construction presented before Theorem~\ref{thm:all}. Let $\HH^n_m$ be the system formed by all these triples in all copies of $H_m$. Because the triples in a given copy use only edges of that copy of $H_m$, and distinct copies are edge-disjoint, triples in distinct copies can share at most one vertex. This immediately implies that $\HH^n_m$ is $\CC\DD$-free. Suppose there is a $\BB$ configuration in $\HH^n_m$. Then the two triples in it that share an edge must come from the same copy of $H_m$ and be of the form $A_i \cup \{b\}$, $A_i \cup \{b'\}$ for some $b \ne b' \in B$. The third triple contains $\{b,b'\}$, which is an edge of the same copy of $H_m$. Thus the third triple could only come from this same copy, but no triple in the construction contains two vertices of $B$. We have shown that $\HH^n_m$ is $\BB\CC\DD$-free. The number of triples in it is \[|\HH^n_m| = {m \choose 2}m \cdot \frac{1}{m(m-1)(m+1)}{n \choose 2} = \frac{1}{2} \frac{m}{m+1} {n \choose 2}.\] When $m$ and $n$ both go to infinity, with $n$ satisfying the above conditions with respect to $m$, namely $n \ge n_0(H_m)$ and $2m(m-1)(m+1) | n-1$, we have $|\HH^n_m| = \frac{1}{4}n^2 \bigl(1 - o(1) \bigr)$. For sufficiently large $n$ that does not satisfy the divisibility condition, we can find an $n'$ that does, $n - 2m(m-1)(m+1) < n' < n$, and take the triple system $\HH^{n'}_m$, still of the required asymptotic size. \qed

\medskip

\noindent {\bf Acknowledgments} \quad
Discussions with Andr\'as Gy\'arf\'as, Ervin Gy\H{o}ri, Gyula O.\ H. Katona, and A. Kostochka
are gratefully acknowledged. In particular, the construction presented in Section~\ref{sec:all} was communicated by Andr\'as Gy\'arf\'as as a conjectured optimal
construction.

\end{document}